\documentclass[3p]{elsarticle}
\makeatletter
\def\ps@pprintTitle{%
   \let\@oddhead\@empty
   \let\@evenhead\@empty
   \let\@oddfoot\@empty
   \let\@evenfoot\@oddfoot
}
\makeatother

\usepackage{amsmath,amsfonts,amsthm,amssymb,eucal} 
\usepackage{graphicx,subfig,tikz} 
\usepackage{color,xcolor} 
\usetikzlibrary{fadings} 
\usetikzlibrary{matrix,arrows.meta} 
\newcommand{\freccia}{-{Stealth[scale=1.5]}} 
\usepackage{appendix}

\newtheorem{thm}{Theorem}[section]

\theoremstyle{definition}

\newtheorem{oss}[thm]{Remark}

\usepackage[norelsize,ruled,vlined,titlenumbered]{algorithm2e}
\SetKwFor{For}{for}{}{end for}
\SetKwRepeat{Do}{do}{while}

\definecolor{celeste}{HTML}{0072BD}
\definecolor{rosso}{HTML}{A2142F}
\newcommand{\cN}{\mathcal{N}}
\newcommand{\cM}{\mathcal{M}}
\newcommand{\cB}{\mathcal{B}}
\newcommand{\cL}{\mathcal{L}}

\newcommand{\cS}{\mathcal{S}}

\newcommand{\cE}{\mathcal{E}}
\newcommand{\RR}{\mathbb{R}}
\newcommand{\NN}{\mathbb{N}}
\newcommand{\fS}{\mathfrak{S}}
\newcommand{\supp}{\text{supp}\,}
\newcommand{\NS}{N$_2$S}
\newcommand{\NNSS}{N$_2$S$_2$}
\renewcommand{\int}{\text{int}}
\newcommand{\SSS}{\mathbb{S}}
\newcommand{\spn}{\text{span}\,}
\newcommand{\diam}{\text{diam}}
\newcommand{\spp}{\text{sep}}

\begin{document}
\begin{frontmatter}
\title{Effective grading refinement for locally linearly independent LR B-splines}

\author{Francesco Patrizi}
\ead{francesco.patrizi@ipp.mpg.de}

\address{Max-Planck-Institut f{\"u}r Plasmaphysik, Boltzmannstra{\ss}e 2, 85748 Garching bei M{\"u}nchen, Germany}

\begin{abstract}
We present a new refinement strategy for locally refined B-splines which ensures the local linear independence of the basis functions. The strategy also guarantees the spanning of the full spline space on the underlying locally refined mesh. The resulting mesh has nice grading properties which grant the preservation of shape regularity and local quasi uniformity of the elements in the refining process. 
\end{abstract}

\begin{keyword}
LR B-splines \sep local linear independence \sep graded meshes \sep adaptive methods.
\end{keyword}
\end{frontmatter}

\section{Introduction}
Locally Refined (LR) B-splines have been introduced in \cite{tor} as generalization of the tensor product B-splines to achieve adaptivity in the discretization process. By allowing local insertions in the underlying  mesh, the approximation efficiency is dramatically improved as one avoids the wasting of degrees of freedom by increasing the number of basis functions only where rapid and large variations occur in the analyzed object. 
Nevertheless, the adoption of LR B-splines for simulation purposes in the Isogeometric Analysis (IgA) framework \cite{iga} is hindered by the risk of linear dependence relations \cite{lindep}. Although a complete characterization of linear independence is still not available, the \emph{local} linear independence of the basis functions is guaranteed when the underlying Locally Refined (LR) mesh has the so-called Non-Nested-Support (\NS) property \cite{bressan1,bressan2}. The local linear independence not only avoids the hurdles of dealing with singular linear systems, but it also improves the sparsity of the matrices when assembling the numerical solution. Furthermore, it allows the construction of efficient quasi-interpolation schemes \cite{N2S2}. Such a strong property of the basis functions is a rarity, or at least it is quite cumbersome to gain, among the technologies used for adaptive IgA. For instance, it is not available for (truncated) hierarchical B-splines \cite{hb,thb} while it can be achieved for PHT-splines \cite{pht} and Analysis-suitable (and dual-compatible) T-splines \cite{ast}, respectively, by imposing reduced regularity and by endorsing a considerable propagation in the refinement \cite{ast1}. 

In this work we present a new refinement strategy to produce LR meshes with the \NS~property. In addition to the local linear independence of the associated LR B-splines, the strategy proposed has two further features: the space spanned coincides with the full space of spline functions and it guarantees  smooth grading in the transitions between coarser and finer regions on the LR meshes produced. The former property boosts the approximation power with respect to the degrees of freedom as the spaces used for the discretization in the IgA context are in general just subsets of the spline space. Such a spanning completeness is more demanding to achieve in terms of meshing constraints and regularity, respectively, for (truncated) hierarchical B-splines and splines over T-meshes \cite{thb1,thb2,bressan2,ast2}. 
The grading properties are instead required to theoretically ensure optimal algebraic rates of convergence in adaptive IgA methods \cite{axioms, b}, even in presence of singularities in the PDE data or solution, similarly to what happens in Finite Element Methods (FEM) \cite{nochetto}. More specifically, the LR meshes generated by the proposed strategy satisfy the requirements listed in the  axioms of adaptivity \cite{axioms} in terms of grading and overall appearance. Such axioms constitute a set of sufficient conditions to guarantee convergence at optimal algebraic rate in adaptive methods. Furthermore, mesh grading has been assumed to prove robust convergence of solvers for linear systems arising in the adaptive IgA framework with respect to mesh size and number of iterations \cite{hendrik}. For these reasons, we have called the strategy Effective Grading (EG) refinement strategy.

The next sections are organized as follows. In Section \ref{preliminaries} we recall the definitions of tensor product meshes and B-splines from a perspective that ease the introduction of LR meshes and LR B-splines. In the second part, we define the \NS~property for the LR meshes and provide the characterization for the local linear independence of the LR B-splines. In Section \ref{EG} we first define the EG strategy and then we prove that it has the \NS~property. The completeness of the space spanned and the grading of the LR meshes are discussed at the end of the section. Finally, in Section \ref{conclusions} we draw the conclusions and present the future research. 

\section{Preliminaries}\label{preliminaries}
In this section we recall the definition of Locally Refined (LR) meshes and B-splines and the conditions ensuring the local linear independence of the latter. We stick to the 2D setting for the sake of simplicity, however, many of the following definitions have a direct generalization to any dimension, see \cite{tor} for details. We assume the reader to be familiar with the definition and main properties of 
B-splines, in particular with the knot insertion procedure. An introduction to this topic can be found, e.g., in the review papers \cite{manni1,manni2} or in the classical books \cite{deboor} and \cite{schumaker}.
\subsection{LR meshes and LR B-splines}
LR meshes and related sets of LR B-splines are constituted simultaneously and iteratively from tensor meshes and sets of tensor B-splines. We therefore start by recalling the latter using a terminology which is proper of the LR B-spline theory. Thereby, we can easily introduce the new concepts by generalizing the tensor case.
A \textbf{tensor (product) mesh} on an axes-aligned rectangular domain $\Omega \subseteq \RR^2$ can be represented as a triplet $\cN = (\cM, \pmb{p}, \mu)$ where $\cM$ is a collection (with repetitions) of \textbf{meshlines}, which are the segments connecting two (and only two) vertices of a rectangular grid on $\Omega$.  $\pmb{p}=(p_1,p_2)$ is a \textbf{bidegree}, that is, a pair of integers in $\NN$, and $\mu:\cM \to \NN^*$ is a map that counts the number of times any meshline $\gamma$ appears in $\cM$. $\mu(\gamma)$ is called \textbf{multiplicity of the meshline} $\gamma$. Furthermore, the following constraints are imposed on $\cM$: 
\begin{itemize}
\item[C1.] $\mu(\gamma_1) = \mu(\gamma_2)$ if $\gamma_1,\gamma_2\in\cM$ are contiguous and aligned,
\item[C2.] $\mu(\gamma) \leq p_1+1$ if $\gamma \in \cM$ is vertical and $\mu(\gamma) \leq p_2+1$ if   $\gamma$ is horizontal. In particular, we say that $\gamma$ has \textbf{full multiplicity} if the equality holds.
\end{itemize}  
A tensor mesh $\cN$ is $\textbf{open}$ if the meshlines on $\partial \Omega$ have full multiplicities. 

Given an open tensor mesh $\cN=(\cM, \pmb{p}, \mu)$, consider another tensor mesh $\cN_B := (\cM_B, \pmb{p}, \mu_B)$ where $\cM_B$ is a sub-collection of meshlines $\cM_B \subseteq \cM$ forming a rectangular grid in a sub-domain $\Omega_B\subseteq \Omega$ of $p_1+2$ vertical lines and $p_2+2$ horizontal lines, where a line is counted $m$ times if the meshlines in it have multiplicity $m$ with respect to $\mu_B$. The multiplicity $\mu_B:\cM_B \to \NN^*$ is such that $\mu_B(\gamma) \leq \mu(\gamma)$ for all $\gamma \in \cM_B$. Such vertical and horizontal lines can be parametrized as $\{x_i\}\times [y_1, y_{p_2+2}]$ and $[x_1, x_{p_1+2}]\times \{y_j\}$ with $\pmb{x}:=(x_i)_{i=1}^{p_1+2}$ and $\pmb{y}=(y_j)_{j=1}^{p_2+2}$ such that $x_i \leq x_{i+1}$ and $y_j \leq y_{j+1}$ and with $x_i, y_j$ appearing $p_1+1$ and $p_2+1$ times at most in $\pmb{x}$ and $\pmb{y}$, respectively, because of the constraint C2 on $\cM$. 
On $\pmb{x}$ and $\pmb{y}$ we can define a \textbf{tensor (product) B-spline}, $B=B[\pmb{x}, \pmb{y}]$. Then, we have that the support of $B$ is $\Omega_B$ and hence $\cN_B$ is a tensor mesh in $\supp B$. We say that $B$ has \textbf{minimal support} on $\cN$ if no line in $\cM\backslash \cM_B$ traverses $\int(\supp B)$ entirely and $\mu_B \equiv \mu$ on the meshlines of $\cM_B$ in the interior of $\supp B$. 
The collection of all the minimal support B-splines on $\cN$ constitutes the \textbf{B-spline set on $\cN$}. If instead $B$ has not minimal support on $\cN$, then there exists a line in $\cM$ entirely traversing $\int(\supp B)$ which either is not in $\cM_B$ or it is in $\cM_B$ but its meshlines have a higher multiplicity with respect to $\mu$ than $\mu_B$. In both cases, such exceeding line corresponds to extra knots either in the $x$- or $y$-direction. One could then express $B$ with B-splines of minimal support on $\cN$ by performing knot insertions. An example of B-spline with no minimal support on a tensor mesh is reported in Figure \ref{nominimal}.
\begin{figure}
\centering
\subfloat[]{
\begin{tikzpicture}[scale=3]
\draw[white] (.4,-.275) node[below]{$x_1^2$};
\filldraw[fill=red!50!white,draw=red, line width=4pt] (.2,.4) -- (1,.4) -- (1,1) -- (.2,1) -- cycle; 
\draw[red,line width=4pt] (.4,.4) -- (.4,1);
\draw[red,line width=4pt] (.8,.4) -- (.8,1);
\draw[red,line width=4pt] (.2,.6) -- (1,.6);
\draw[red,line width=4pt] (.2,.8) -- (1,.8);
\draw[step=.2] (0,0) grid (1.2,1.2);
\draw (.2,0) node[below]{$x_1$};
\draw (.4,0) node[below]{$x_2$};
\draw (.8,0) node[below]{$x_3$};
\draw (1,0) node[below]{$x_4$};
\draw (.6,.025) node[below]{$\hat{x}$};
\draw (0,.4) node[left]{$y_1$};
\draw (0,.6) node[left]{$y_2$};
\draw (0,.8) node[left]{$y_3$};
\draw (0,1) node[left]{$y_4$};
\end{tikzpicture}}\quad
\subfloat[]{
\begin{tikzpicture}[scale=3]
\filldraw[fill=green!50!black,opacity=.5,draw=green!50!black, line width=4pt] (.2,.4) -- (.8,.4) -- (.8,1) -- (.2,1) -- cycle; 
\draw[green!50!black,line width=4pt,opacity=.5] (.4,.4) -- (.4,1);
\draw[green!50!black,line width=4pt,opacity=.5] (.6,.4) -- (.6,1);
\draw[green!50!black,line width=4pt,opacity=.5] (.2,.6) -- (.8,.6);
\draw[green!50!black,line width=4pt,opacity=.5] (.2,.8) -- (.8,.8);

\filldraw[fill=orange!90!black,opacity=.5,draw=orange!90!black, line width=4pt] (.4,.4) -- (1,.4) -- (1,1) -- (.4,1) -- cycle; 
\draw[orange!90!black,line width=4pt,opacity=.5] (.8,.4) -- (.8,1);
\draw[orange!90!black,line width=4pt,opacity=.5] (.6,.4) -- (.6,1);
\draw[orange!90!black,line width=4pt,opacity=.5] (.4,.6) -- (1,.6);
\draw[orange!90!black,line width=4pt,opacity=.5] (.4,.8) -- (1,.8);
\draw[step=.2] (0,0) grid (1.2,1.2);
\draw (.2,0) node[below]{$x_1^1$};
\draw (.4,0) node[below]{$x_2^1$};
\draw (.8,0) node[below]{$x_4^1$};
\draw (1,0) node[below]{$x_4^2$};
\draw (.6,0) node[below]{$x_3^1$};
\draw (.4,-.15) node[below]{\rotatebox{90}{$=$}};
\draw (.8,-.15) node[below]{\rotatebox{90}{$=$}};
\draw (.6,-.15) node[below]{\rotatebox{90}{$=$}};
\draw (.4,-.275) node[below]{$x_1^2$};
\draw (.8,-.275) node[below]{$x_3^2$};
\draw (.6,-.275) node[below]{$x_2^2$};
\draw (0,.4) node[left]{$y_1$};
\draw (0,.6) node[left]{$y_2$};
\draw (0,.8) node[left]{$y_3$};
\draw (0,1) node[left]{$y_4$};
\end{tikzpicture}
}
\caption{Example of B-spline with no minimal support on a tensor mesh. Let us consider the tensor mesh $\cN=(\cM,(2,2),1)$ as in figure (a). Let also $B = B[\pmb{x}, \pmb{y}]$  be the B-spline  of bidegree $(2,2)$ whose knot vectors are $\pmb{x}=(x_i)_{i=1}^4, \pmb{y}=(y_j)_{j=1}^4$ and whose support and tensor mesh $\cN_B = (\cM_B, (2,2), 1)$ are highlighted in figure (a). $B$ has not minimal support on $\cN$ as the vertical line placed at value $\hat{x}$ is traversing $\supp B$ entirely while its meshlines in $\supp B$ are not contained in $\cM_B$. However, by knot insertion of $\hat{x}$ in $\pmb{x}$ we can express $B$ in terms of two minimal support B-splines on $\cN$, $B[\pmb{x}^1,\pmb{y}]$ and $B[\pmb{x}^2,\pmb{y}]$, with $\pmb{x}^1=(x_i^1)_{i=1}^4, \pmb{x}^2=(x_i^2)_{i=1}^4$. The supports of the latter partially overlap horizontally and are represented in figure (b).}\label{nominimal}
\end{figure}
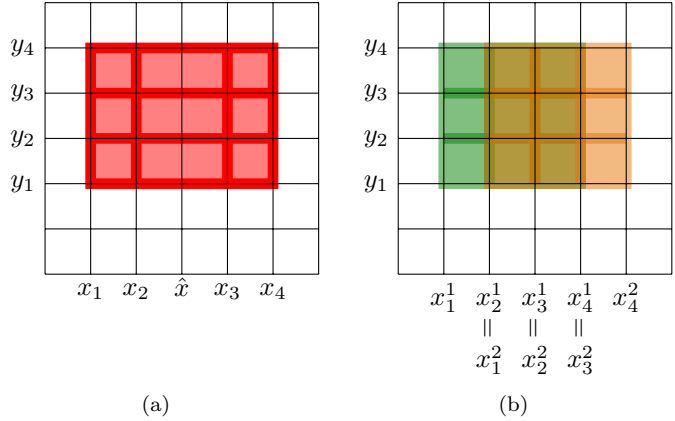

Given now an open tensor mesh $\cN=(\cM, \pmb{p}, \mu)$ and the corresponding B-spline set $\cB$, assume that we either
\begin{itemize}
\item[R1.] raise by one the multiplicity of a set of contiguous and colinear meshlines in $\cM$, which, however, still has to satisfy the constraints C1--C2,
\item[R2.] insert a new axis-aligned line $\gamma$ with endpoints on $\cM$, traversing the support of at least one B-spline $B \in \cB$, and extend $\mu$ to the segments connecting the intersection points of $\gamma$ and $\cM$, by setting it equal to 1 for such new meshlines.
\end{itemize}

Let $\cM'$ be the new collection of meshlines and $\mu'$ be the multiplicity for $\cM'$. By construction,  there exists at least one B-spline $B \in \cB$ that does not have minimal support on $\cN' = (\cM', \pmb{p}, \mu')$. By performing knot insertions we can, however, replace $B$ in the collection $\cB$ with B-splines of minimal support on $\cN'$. This creates a new set $\cB'$ of B-splines of minimal support defined on $\cN'$. We are now ready to define (recursively) LR meshes and LR B-splines.

An \textbf{LR mesh} on $\Omega$ is a triplet $\cN' = (\cM',\pmb{p},\mu')$ which either is a tensor mesh or it is obtained by applying the procedure R1 or R2 to $\cN=(\cM,\pmb{p},\mu)$ which, in turn, is an LR mesh. The \textbf{LR B-spline set} $\cB'$ on $\cN'$ is the B-spline set on $\cN'$ if the latter is a tensor mesh or, in case $\cN'$ is not a tensor mesh, it is obtained via knot insertions from the LR B-spline set $\cB$ defined on $\cN$.

In other words, we refine a coarse tensor mesh by inserting new lines (which possibly can have an endpoint in the interior of $\Omega$), one at a time, or by raising the multiplicity of a line already on the mesh. On the initial tensor mesh we consider the tensor B-splines and whenever a B-spline in our collection has no longer minimal support during the mesh refinement process, we replace it by using the knot insertion procedure. The LR B-splines will be the final set of B-splines produced by this algorithm. In Figure \ref{R2} we illustrate the evolution of an LR B-spline throughout such process.
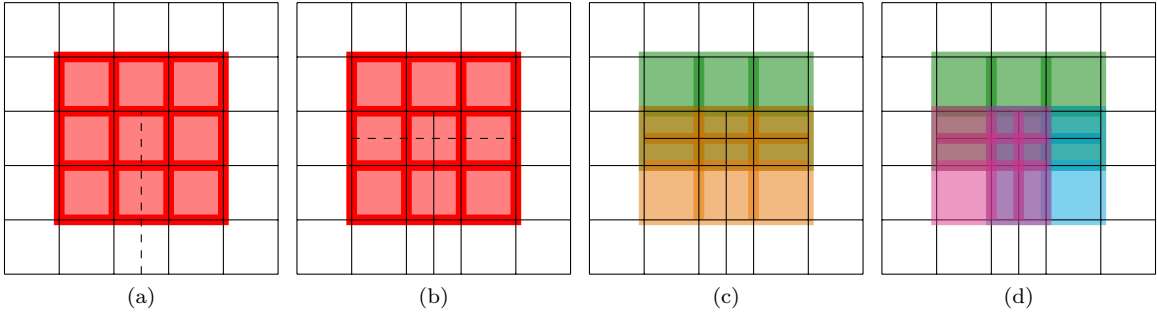
\begin{figure}
\centering
\subfloat[]{
\begin{tikzpicture}[scale=3.6]
\filldraw[fill=red!50!white,draw=red, line width=4pt] (.2,.4) -- (.8,.4) -- (.8,1) -- (.2,1) -- cycle; 
\draw[red,line width=4pt] (.4,.4) -- (.4,1);
\draw[red,line width=4pt] (.6,.4) -- (.6,1);
\draw[red,line width=4pt] (.2,.6) -- (.8,.6);
\draw[red,line width=4pt] (.2,.8) -- (.8,.8);
\draw[step=.2] (0,.2) grid (1,1.2);
\draw[dashed] (.5,.2) -- (.5,.8);
\end{tikzpicture}
}
\subfloat[]{
\begin{tikzpicture}[scale=3.6]
\filldraw[fill=red!50!white,draw=red, line width=4pt] (.2,.4) -- (.8,.4) -- (.8,1) -- (.2,1) -- cycle; 
\draw[red,line width=4pt] (.4,.4) -- (.4,1);
\draw[red,line width=4pt] (.6,.4) -- (.6,1);
\draw[red,line width=4pt] (.2,.6) -- (.8,.6);
\draw[red,line width=4pt] (.2,.8) -- (.8,.8);
\draw[step=.2] (0,.2) grid (1,1.2);
\draw (.5,.2) -- (.5,.8);
\draw[dashed] (.2,.7) -- (.8,.7);
\end{tikzpicture}
}
\subfloat[]{
\begin{tikzpicture}[scale=3.6]
\filldraw[fill=green!50!black,opacity=.5,draw=green!50!black, line width=4pt] (.2,.6) -- (.8,.6) -- (.8,1) -- (.2,1) -- cycle; 
\draw[green!50!black,line width=4pt,opacity=.5] (.4,.6) -- (.4,1);
\draw[green!50!black,line width=4pt,opacity=.5] (.6,.6) -- (.6,1);
\draw[green!50!black,line width=4pt,opacity=.5] (.2,.7) -- (.8,.7);
\draw[green!50!black,line width=4pt,opacity=.5] (.2,.8) -- (.8,.8);

\filldraw[fill=orange!90!black,opacity=.5,draw=orange!90!black, line width=4pt] (.2,.4) -- (.8,.4) -- (.8,.8) -- (.2,.8) -- cycle; 
\draw[orange!90!black,line width=4pt,opacity=.5] (.4,.4) -- (.4,.8);
\draw[orange!90!black,line width=4pt,opacity=.5] (.6,.4) -- (.6,.8);
\draw[orange!90!black,line width=4pt,opacity=.5] (.2,.7) -- (.8,.7);
\draw[orange!90!black,line width=4pt,opacity=.5] (.2,.6) -- (.8,.6);
\draw[step=.2] (0,.2) grid (1,1.2);
\draw (.5,.2) -- (.5,.8);
\draw (.2,.7) -- (.8,.7);
\end{tikzpicture}
}
\subfloat[]{
\begin{tikzpicture}[scale=3.6]
\filldraw[fill=green!50!black,opacity=.5,draw=green!50!black, line width=4pt] (.2,.6) -- (.8,.6) -- (.8,1) -- (.2,1) -- cycle; 
\draw[green!50!black,line width=4pt,opacity=.5] (.4,.6) -- (.4,1);
\draw[green!50!black,line width=4pt,opacity=.5] (.6,.6) -- (.6,1);
\draw[green!50!black,line width=4pt,opacity=.5] (.2,.7) -- (.8,.7);
\draw[green!50!black,line width=4pt,opacity=.5] (.2,.8) -- (.8,.8);

\filldraw[fill=cyan!90!black,opacity=.5,draw=cyan!90!black, line width=4pt] (.4,.4) -- (.8,.4) -- (.8,.8) -- (.4,.8) -- cycle; 
\draw[cyan!90!black,line width=4pt,opacity=.5] (.6,.4) -- (.6,.8);
\draw[cyan!90!black,line width=4pt,opacity=.5] (.5,.4) -- (.5,.8);
\draw[cyan!90!black,line width=4pt,opacity=.5] (.4,.7) -- (.8,.7);
\draw[cyan!90!black,line width=4pt,opacity=.5] (.4,.6) -- (.8,.6);
\draw[step=.2] (0,.2) grid (1,1.2);
\draw (.5,.2) -- (.5,.8);
\draw (.2,.7) -- (.8,.7);

\filldraw[fill=magenta!90!black,opacity=.5,draw=magenta!90!black, line width=4pt] (.2,.4) -- (.6,.4) -- (.6,.8) -- (.2,.8) -- cycle; 
\draw[magenta!90!black,line width=4pt,opacity=.5] (.4,.4) -- (.4,.8);
\draw[magenta!90!black,line width=4pt,opacity=.5] (.5,.4) -- (.5,.8);
\draw[magenta!90!black,line width=4pt,opacity=.5] (.2,.7) -- (.6,.7);
\draw[magenta!90!black,line width=4pt,opacity=.5] (.2,.6) -- (.6,.6);
\end{tikzpicture}
}
\caption{Evolution of an LR B-spline throughout the refinement process of a tensor mesh. Consider the tensor mesh $\cN=(\cM,(2,2),1)$ reported in figure (a). Let $B[\pmb{x},\pmb{y}]$ be the minimal support B-spline whose support and tensor mesh are highlighted in figure (a). Let us insert a first vertical line in $\cN$ (dashed in figure (a)). This line does not traverses $\supp B$, hence $B$ is preserved in the B-spline set on the new LR mesh, as shown in figure (b). We then insert an horizontal line (dashed in figure (b)). This time the line is traversing $\supp B$ and $B[\pmb{x},\pmb{y}]$ is replaced by the B-splines $B[\pmb{x},\pmb{y}^1]$ and $B[\pmb{x},\pmb{y}^2]$ involved in the knot insertion. In figure (c) we see the supports and tensor meshes of the latter on the new LR mesh. In particular we see that $B[\pmb{x},\pmb{y}^1]$ (the bottom B-spline in figure (c)) has not minimal support on the LR mesh as there is a vertical line traversing its support without being part of its tensor mesh. Thus $B[\pmb{x},\pmb{y}^1]$ is replaced as well, via knot insertion, by two other B-splines $B[\pmb{x}^1,\pmb{y}^1], B[\pmb{x}^2, \pmb{y}^1]$. Therefore, in the end, we move from $B[\pmb{x}, \pmb{y}]$, on the tensor mesh, to $B[\pmb{x}^1, \pmb{y}^1], B[\pmb{x}^2,\pmb{y}^1], B[\pmb{x}, \pmb{y}^2]$ on the final LR mesh. The supports and tensor meshes of the latter are represented in figure (d).}\label{R2}
\end{figure}

We conclude this section with a short list of remarks:
\begin{itemize}
\item In general the mesh refinement process producing a given LR mesh is not unique, as the insertion ordering can often be changed. However, the final LR B-spline set is well defined because independent of such insertion ordering, as proved in \cite[Theorem 3.4]{tor}.
\item The LR B-spline set is in general only a subset of the set of minimal support B-spline defined on the LR mesh, although the two sets coincide on the initial tensor mesh. When inserting new lines the LR B-splines are the result of the knot insertion procedure, applied to LR B-splines defined on the previous LR mesh, while some minimal support B-splines could be created from scratch on the new LR mesh. Further details and examples can found in \cite[Section 5]{lindep}.
\item We have introduced LR meshes and LR B-splines starting from open tensor meshes and related sets of tensor B-splines. It is actually not necessary that the initial tensor mesh is open, as long as it is possible to define at least one tensor B-spline on it. The openness was assumed indeed to verify this requirement.
\item In the next sections, we always consider tensor and LR meshes with boundary meshlines of full multiplicity and internal meshlines of multiplicity 1, if not specified otherwise. In particular, this means that we update the LR meshes and LR B-spline sets  only by performing the procedure R2.
\end{itemize}

\subsection{Local linear independence and N$_2$S-property}
The LR B-splines coincide with the tensor B-splines when the underlying LR mesh is a tensor mesh and in general the formulation of LR B-splines remains broadly similar to that of tensor B-splines even though the former address local refinements. As a consequence, in addition to making them one of the most elegant extensions to achieve adaptivity, this similarity implies that many of the B-spline properties are preserved by the LR B-splines. For example, they are non-negative, have minimal support, are piecewise polynomials and can be expressed by the LR B-splines on finer LR meshes using non-negative coefficients (provided by the knot insertion procedure). Furthermore, it is possible to scale them by means of positive weights so that they also form a partition of unity, see \cite[Section 7]{tor}. 

However, as opposed to tensor B-splines, they could be not locally linearly independent. Actually, the set of LR B-splines can even be linearly dependent (examples can be found in \cite{tor,lindep,N2S2}). 

Nevertheless, in \cite{bressan1,bressan2} a characterization of the local linear independence of the LR B-splines has been provided in terms of meshing constraints leading to particular arrangements of the LR B-spline supports on the LR mesh. In this section we recall such characterization.

First of all, we introduce the concept of nestedness. Given an LR mesh $\cN=(\cM,\pmb{p},\mu)$, let $B_1, B_2$ be two different LR B-splines defined on $\cN$. We say that $B_2$ \textbf{is nested in} $B_1$ if
\begin{itemize}
\item $\supp B_2 \subseteq \supp B_1$,
\item $\mu_{B_2}(\gamma) \leq \mu_{B_1}(\gamma)$ for all the meshlines $\gamma$ of $\cM$ in $\partial \supp B_1 \cap \partial \supp B_2$. 
\end{itemize}
An LR mesh where no LR B-spline is nested is said to have the \textbf{Non-Nested-Support property}, or in short the \textbf{\NS~property}. Figure \ref{nested} shows an example of an LR B-spline nested in another. 
\begin{figure}[t!]
\centering
\subfloat[]{
\begin{tikzpicture}[scale=3.2]
\draw (0,0) -- (1,0) -- (1,1) -- (0,1) -- cycle;
\draw (.02,0) -- (.02,1);
\draw (.33, 0) -- (.33, 1);
\draw (.66, 0) -- (.66, 1);
\draw (0, .3) -- (1, .3);
\draw (0,.7) -- (1, .7);
\draw (.5, 0) -- (.5,.7);
\draw (0,.5) -- (.66,.5);
\end{tikzpicture}
}\qquad
\subfloat[]{
\begin{tikzpicture}[scale=3.2]
\filldraw[fill=red!50, draw = red, line width=4pt] (.02,0) -- (1,0) -- (1,1) -- (.02,1) -- cycle;
\draw[red, line width=4pt] (.02,.33) -- (1,.33);
\draw[red, line width=4pt] (.02,.66) -- (1,.66);
\draw[red, line width=4pt] (.33,0) -- (.33,1);
\draw[red, line width=4pt] (.66,0) -- (.66,1);
\draw (0,0) -- (1,0) -- (1,1) -- (0,1) -- cycle;
\draw (.02,0) -- (.02,1);
\draw (.33, 0) -- (.33, 1);
\draw (.66, 0) -- (.66, 1);
\draw (0, .33) -- (1, .33);
\draw (0,.66) -- (1, .66);
\draw (.5, 0) -- (.5,.66);
\draw (0,.5) -- (.66,.5);
\end{tikzpicture}
}\qquad
\subfloat[]{
\begin{tikzpicture}[scale=3.2]
\filldraw[fill=green!50!black!50, draw = green!50!black, line width=4pt] (.02,0) -- (.66,0) -- (.66,.66) -- (.02,.66) -- cycle;
\draw[green!50!black, line width=4pt] (.02,.33) -- (.66,.33);
\draw[green!50!black, line width=4pt] (.02,.5) -- (.66,.5);
\draw[green!50!black, line width=4pt] (.5,0) -- (.5,.66);
\draw[green!50!black, line width=4pt] (.33,0) -- (.33,.66);
\draw (0,0) -- (1,0) -- (1,1) -- (0,1) -- cycle;
\draw (.02,0) -- (.02,1);
\draw (.33, 0) -- (.33, 1);
\draw (.66, 0) -- (.66, 1);
\draw (0, .33) -- (1, .33);
\draw (0,.66) -- (1, .66);
\draw (.5, 0) -- (.5,.66);
\draw (0,.5) -- (.66,.5);
\end{tikzpicture}
}\qquad
\subfloat[]{
\begin{tikzpicture}[scale=3.2]
\filldraw[fill=cyan!50, draw = cyan, line width=4pt] (0,0) -- (.5,0) -- (.5,.66) -- (0,.66) -- cycle;
\draw[cyan, line width=4pt] (0,.33) -- (.5,.33);
\draw[cyan, line width=4pt] (0,.5) -- (.5,.5);
\draw[cyan, line width=4pt] (.33,0) -- (.33,.66);
\draw[cyan, line width=4pt] (.02,0) -- (.02,.66);
\draw (0,0) -- (1,0) -- (1,1) -- (0,1) -- cycle;
\draw (.02,0) -- (.02,1);
\draw (.33, 0) -- (.33, 1);
\draw (.66, 0) -- (.66, 1);
\draw (0, .33) -- (1, .33);
\draw (0,.66) -- (1, .66);
\draw (.5, 0) -- (.5,.66);
\draw (0,.5) -- (.66,.5);
\end{tikzpicture}
}\caption{Example of nested LR B-splines on the LR mesh $\cN=(\cM,(2,2),\mu)$ shown in (a). All the meshlines of $\cM$ have multiplicity 1 except those in the left edge, highlighted with a double line, which have multiplicity 2.  In (b)--(d) three LR B-splines $B_1,B_2,B_3$ defined on $\cN$, represented by means of their supports and tensor meshes. All the meshlines in $\cM_{B_1},\cM_{B_2}$ and $\cM_{B_3}$ have multiplicity 1 except those on the left edge in $\cM_{B_3}$ which have multiplicity 2. Therefore, $B_2$ is nested in $B_1$ while $B_3$ is not nested neither in $B_2$ nor $B_3$, despite that $\supp B_3 \subseteq \supp B_2$ and $\supp B_3 \subseteq\supp B_1$, because the shared meshlines in the left edge of $\supp B_3$, $\supp B_2$ and $\supp B_1$ have multiplicity 2 in $\cM_{B_3}$ and multiplicity 1 in $\cM_{B_2}$ and $\cM_{B_1}$.}\label{nested}
\end{figure}
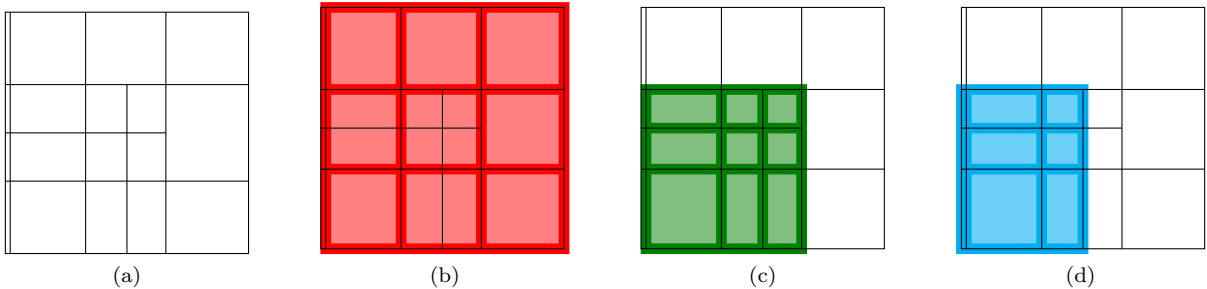

The next result, from \cite[Theorem 4]{bressan2}, relates the local linear independence of the LR B-splines to the \NS~property of the LR mesh. In order to present it, we recall that given an LR mesh $\cN=(\cM,\pmb{p},\mu)$, $\cM$ induces a \textbf{box-partition} of $\Omega$, that is, a collection of axes-aligned rectangles, called \textbf{boxes}, with disjoint interiors covering $\Omega$. Hereafter, we will just call them boxes of $\cM$, with an abuse of notation, instead of boxes in the box-partition induced by $\cM$. 

\begin{thm}\label{loclinind}
Let $\cN=(\cM, \pmb{p}, \mu)$ be an LR mesh and let $\cL$ be the related LR B-spline set. The following statements are equivalent.
\begin{enumerate}
\item The elements of $\cL$ are locally linearly independent.
\item $\cN$ has the \NS~property.
\item Any box $\beta$ of $\cM$ is contained in exactly $(p_1+1)(p_2+1)$ LR B-spline supports, that is,
$$
\#\{B \in \cL \,:\, \supp B \supseteq \beta\} = (p_1+1)(p_2+1).
$$
\item The LR B-splines in $\cL$ form a partition of unity, without the use of scaling weights.
\end{enumerate}
\end{thm}

In the next section we present an algorithm to construct LR meshes with the \NS~property. The resulting LR meshes will furthermore show a nice gradual grading from coarser regions to finer regions, which avoids the thinning in some direction of the box sizes and the placing of small boxes side by side with large boxes. 
\section{The Effective grading refinement strategy}\label{EG}
In this section we present a refinement strategy to generate LR meshes with the \NS~property. We call it \textbf{Effective Grading (EG) refinement strategy} as the finer regions smoothly fade towards the coarser regions in the resulting LR meshes. 

To the best of our knowledge, two other strategies have been proposed to build LR meshes with the \NS~property so far: the \textbf{Non-Nested-Support-Structured (\NNSS) mesh refinement} \cite{N2S2} and the \textbf{Hierarchical Locally Refined (HLR) mesh refinement} \cite{bressan2}. The \NNSS~mesh refinement is a \textbf{function-based} refinement strategy, which means that at each iteration we refine those LR B-splines contributing more to the approximation error, in some norm. The \NNSS~mesh strategy does not require any condition on the LR B-splines selected for refinement to ensure the \NS~property of the resulting LR meshes. On the other hand, no grading has been proved on the final LR meshes and skinny elements may be present on them. 
The HLR refinement is instead a \textbf{box-based} strategy, which means that at each iteration the region to refine is identified by those boxes, in the box-partition induced by the LR mesh, in which a larger error is committed, in some norm. The HLR strategy produces nicely graded LR meshes but it requires that the regions to be refined and the maximal resolution have to be chosen a priori to ensure the \NS~property. Usually one does not know in advance where the error will be large and how fine the mesh has to be to reduce the error under a certain tolerance. Therefore, the conditions for the \NS~property constitute a drawback for the adoption of the HLR strategy in many practical purposes.

The EG refinement is a box-based strategy providing LR meshes very similar to those that one gets with the HLR strategy, when fixing the refinement regions and the number of iterations. As we shall show, the LR meshes generated will always have the \NS~property, with no requirements or assumptions.
\subsection{Preliminary observations and generalized shadow map}\label{sec:EGpre}
In order to introduce the strategy, we need some preliminary considerations on the LR meshes produced by the algorithm. For the sake of simplicity, we assume our domain $\Omega$ to be a square $\Omega = [a,b]^2 \subseteq \RR^2$. Given an LR mesh $\cN =(\cM, \pmb{p},\mu)$ in $\Omega$, generated by several applications of EG strategy, and a box $\beta$ in $\cM$, we define the \textbf{diameter of} $\beta$, denoted by $\diam(\beta)$, as the length of the diagonal of $\beta$. As we shall show in Section \ref{sec:EGdef}, the boxes in $\cM$ are either squares or rectangles with one side twice the other. Furthermore, such boxes are obtained by halving boxes in the previous mesh in one of the two directions, i.e., square boxes are refined in rectangles and rectangular boxes are refined in square boxes. In particular, the width $L$ of the longest side of any given box of $\cM$ has expression
$$
L = \frac{b-a}{2^q} \quad\text{for some }q\in \NN.
$$
This means that $\beta$ is a square box in $\cM$ if and only if
$$
\frac{\diam(\beta)}{L} = \sqrt{2}\text{, that is, if and only if } \frac{(b-a)^2}{\diam(\beta)^2} = 2^{2q-1}.
$$
Whereas, $\beta$ is a rectangular box in $\cM$ if and only if $$
\frac{\diam(\beta)}{L} = \sqrt{\frac{3}{2}}\text{, that is, if and only if } \frac{3(b-a)^2}{\diam(\beta)^2} = 2^{2q+1}.
$$
Hence, given $\diam(\beta)$ we can understand if $\beta$ is a square or a rectangular box by looking at $\frac{3(b-a)^2}{\diam(\beta)^2} \mod 3$:
$$
\beta\text{ is a }\left\{\begin{array}{l}
\text{square box }\iff \frac{3(b-a)^2}{\diam(\beta)^2} \equiv 0 \mod 3,\\\\
\text{rectangular box }\iff \frac{3(b-a)^2}{\diam(\beta)^2} \not\equiv 0 \mod 3,
\end{array}\right.
$$
with the only exception of the square box $\beta = \Omega$, for which $q=0$ and $\frac{3(b-a)^2}{\diam(\beta)^2} = \frac{3}{2}$.

There are two variants of the EG strategy, the ``\emph{Horizontal-major}'' and the ``\emph{Vertical-major}''. In the Horizontal-major version, the boxes of the mesh, at any iteration, are squares or rectangles of width twice the height. Hence, square boxes are refined by halving them horizontally, while rectangular boxes are refined by halving them vertically. In the Vertical-major case it is the opposite: squares are refined in rectangles of height twice the width, by halving them vertically, and rectangular boxes are refined in square boxes, by halving them horizontally. In Figure \ref{fig:HVvariants} we compare the two variants by refining along the same ``bean curve'', using bidegree $(2,2)$ and $8$ levels of refinement in each direction. 
\begin{figure}
\centering
\subfloat[$8001$ LR B-splines]{\includegraphics[width=.333\textwidth]{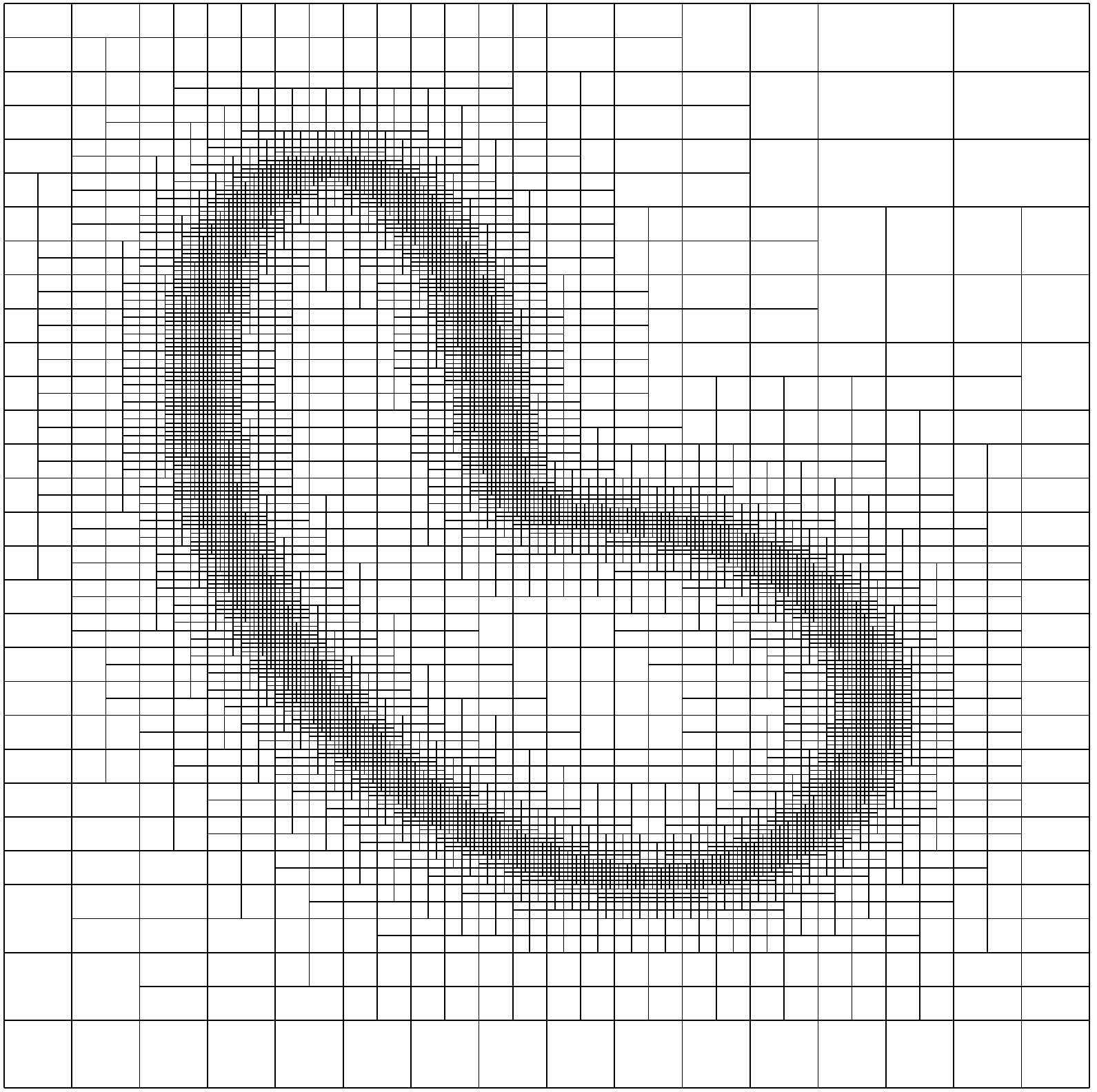}}\qquad
\subfloat[$7821$ LR B-splines]{\includegraphics[width=.333\textwidth]{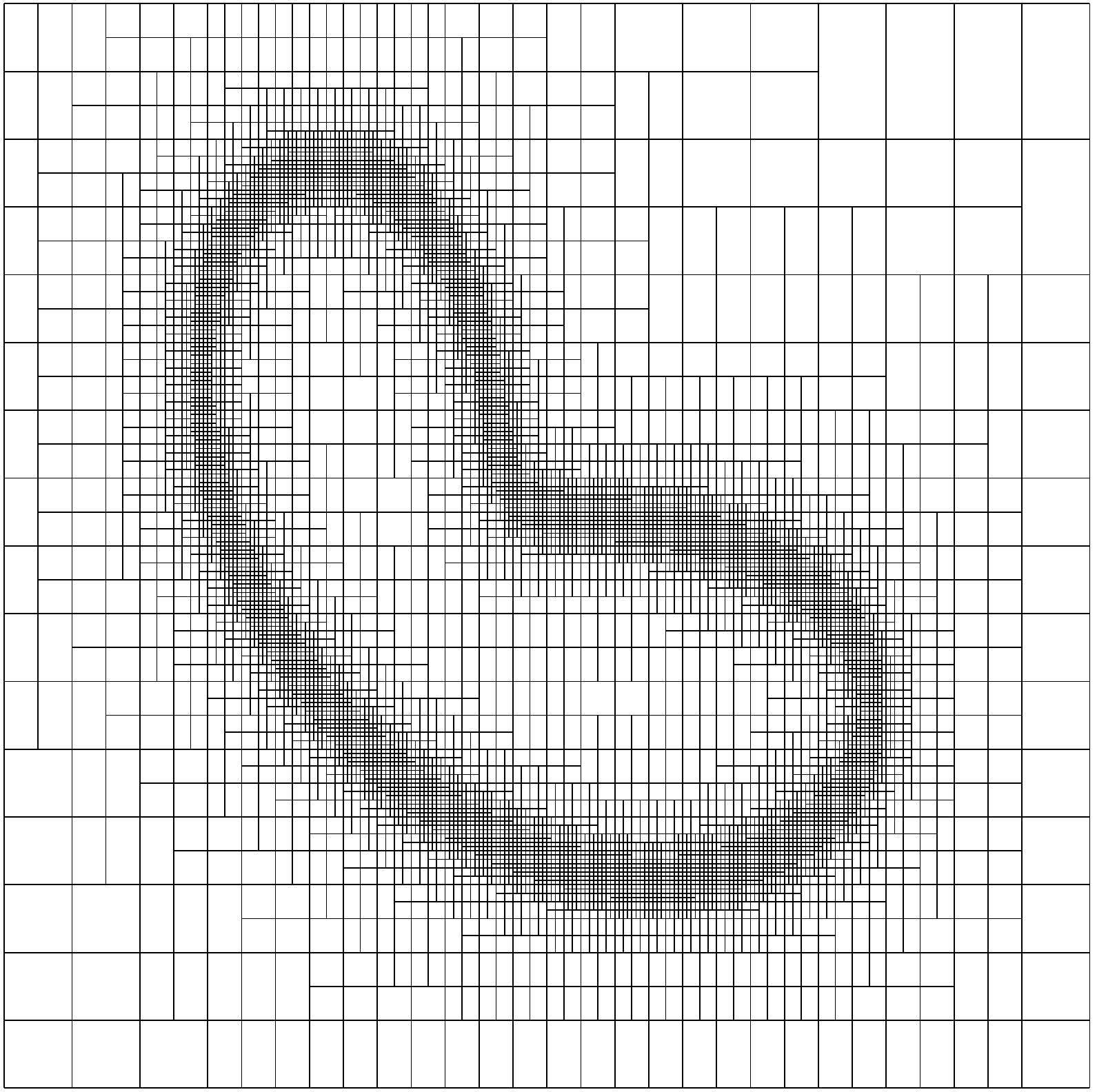}}
\caption{Comparison of the Horizontal-major and Vertical-major versions of the EG strategy. We consider bidegree $(2,2)$ and refine along a ``bean curve'' with $8$ levels of resolution in each direction. In (a) the Horizontal-major variant of the strategy and in (b) the Vertical-major variant of the strategy. In the former the rectangular boxes have width twice the height. In the latter the rectangular boxes have height twice the width. We also show the number of LR B-splines defined in the meshes. Obviously the two cardinalities are different because the two meshes are not equal after flipping the axes, due to the asymmetry of the curve.}\label{fig:HVvariants}
\end{figure}

In the description of the EG strategy in Section \ref{sec:EGdef}, we will just use the verb ``to halve'', without specifying the direction, to treat the two variants at the same time. 

Let $\beta$ be a square box in the mesh of diameter $\diam(\beta) = d$. $\beta$ has been obtained by halving a box of diameter $$d' = \sqrt{\frac{5}{2}}d.$$ Instead, if $\beta$ is a rectangular box, it has been obtained by halving a box of diameter $$d' = 2\sqrt{\frac{2}{5}}d.$$ In the description of the EG strategy in Section \ref{sec:EGdef} we will denote by $s$ the scaling factor to express $d'$ in terms of $d$, i.e., $d' = sd$, independently of the shape of the box at hand. 

Finally, we introduce the \textbf{generalized shadow map} of a set $A$ in $\Omega$. As opposed to the shadow map \cite[Definition 10]{bressan2} which is defined for tensor meshes, the generalized shadow map can be applied in locally refined meshes. The latter is consistent with the former, that is, the two maps are equivalent, when the underlying LR mesh is a tensor mesh and $A$ consists of a bunch of boxes of the mesh, as we shall show in the appendix of this paper. Given an LR mesh $\cN=(\cM, \pmb{p},\mu)$ and a set $A$ in $\Omega$, the generalized shadow map of $A$ in $\cN$ defines a superset of $A$ which is larger only along one of the two directions, as follows. We present only the horizontal shadow map for briefness, the procedure for the vertical is analogous. For the sake of simplicity, let us assume first that $A$ has only one connected component. For any point $\pmb{q} \in \partial A$ we consider the two horizontal half-lines from $\pmb{q}$, $r^1$ and $r^2$.
Let $\pmb{q}_1^i, \ldots, \pmb{q}_{N_i}^i$ be the intersection points of $r^i$ with the vertical meshlines of $\cM$ (counting their multiplicites), where $\pmb{q}_1^i$ is the closest to $\pmb{q}$ and $\pmb{q}_{N_i}^i$ the farthest. In particular, note that if $\pmb{q}$ lies on a vertical line of $\cM$, then $\pmb{q}_1^i = \pmb{q}$. We define 
\begin{equation}\label{q*}
\pmb{q}_*^i := \left\{\begin{array}{ll}
\pmb{q}_{p_1+1}^i & \text{if }N_i\geq p_1+1,\\\\
\pmb{q}_{N_i}^i & \text{otherwise.}
\end{array}\right.
\end{equation}
The (horizontal) generalized shadow of $A$ with respect to $\cN$, denoted by $\cS A$, are the boxes of $\cM$ intersecting the points in the segments $\overline{\pmb{q}_*^1\pmb{q}_*^2}$ for $\pmb{q} \in \partial A$ or the points in $A$, that is, 
$$
\cS A := \left\{\beta \text{ box of }\cM\,:\, \beta \cap \left(A \cup \left(\bigcup_{\pmb{q} \in \partial A} \overline{\pmb{q}_*^1\pmb{q}_*^2}\right)\right)\neq \emptyset\right\}.
$$
If $A$ has more connected components, $A_1, \ldots, A_M$, then the generalized shadow $\cS A$ will be the union of the generalized shadows of the connected components:
$$
\cS A := \bigcup_{j=1}^M \cS A_j.
$$
\begin{figure}
\centering
\subfloat[]{
\includegraphics[width=.22\textwidth]{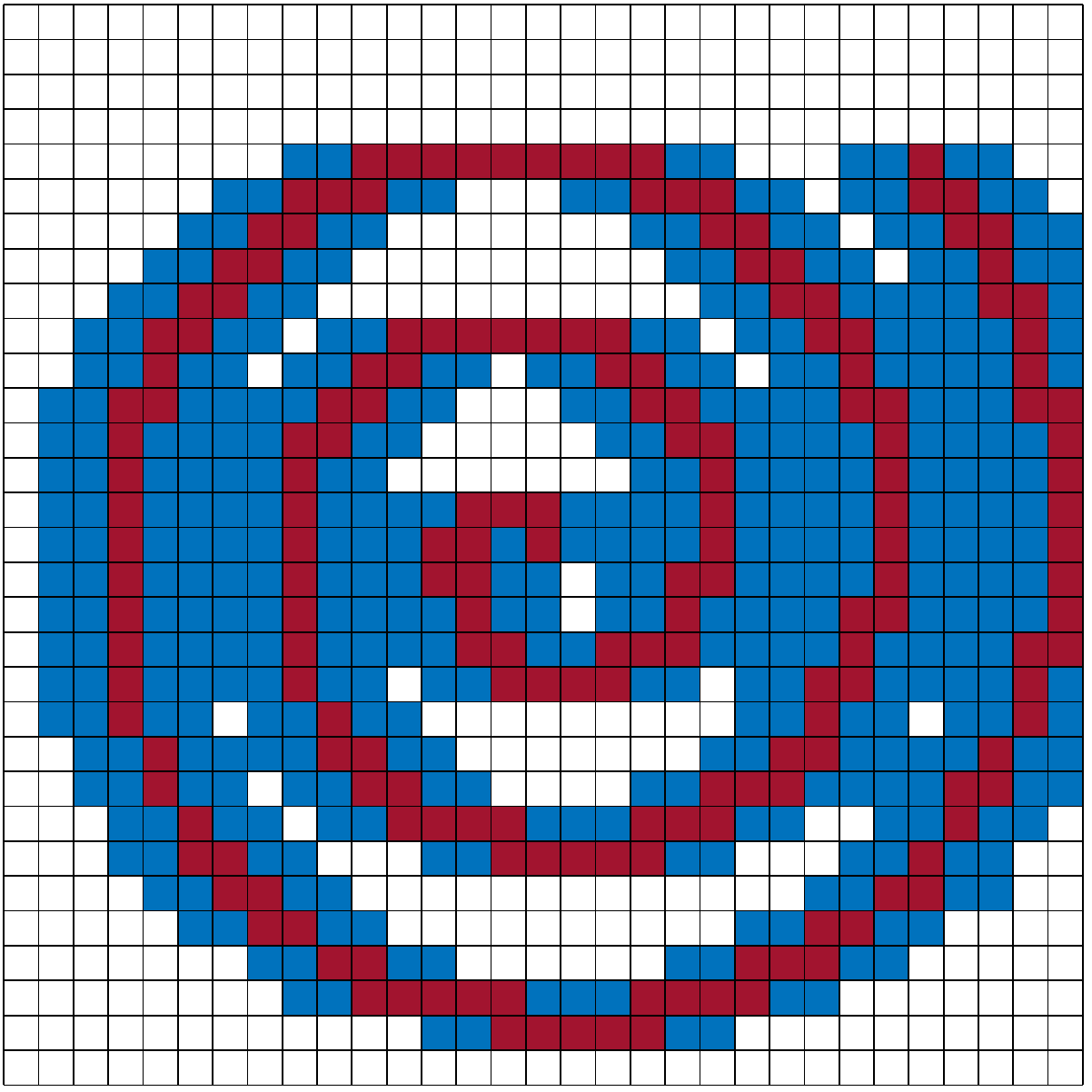}}\quad
\subfloat[]{
\includegraphics[width=.22\textwidth]{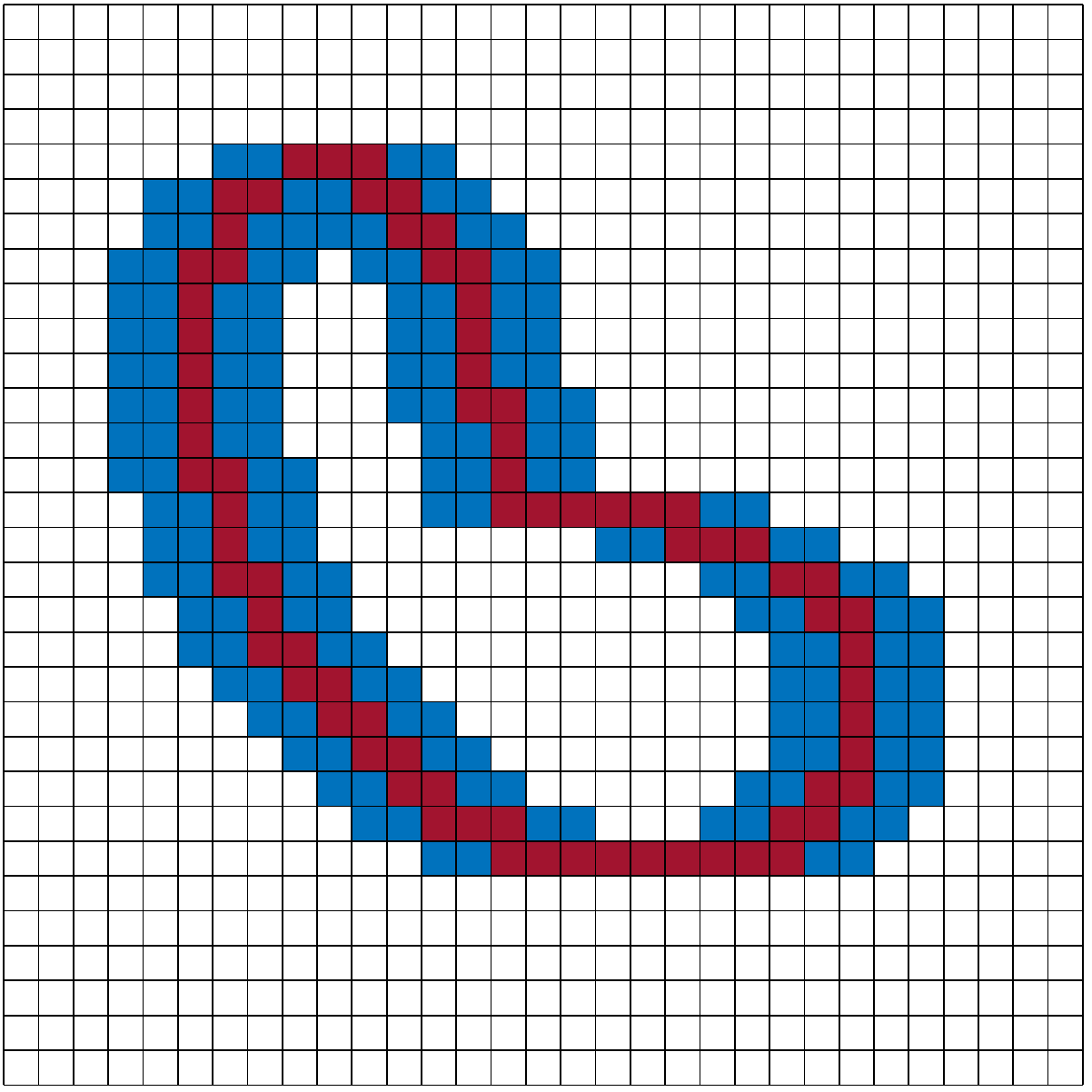}}\quad
\subfloat[]{
\includegraphics[width=.22\textwidth]{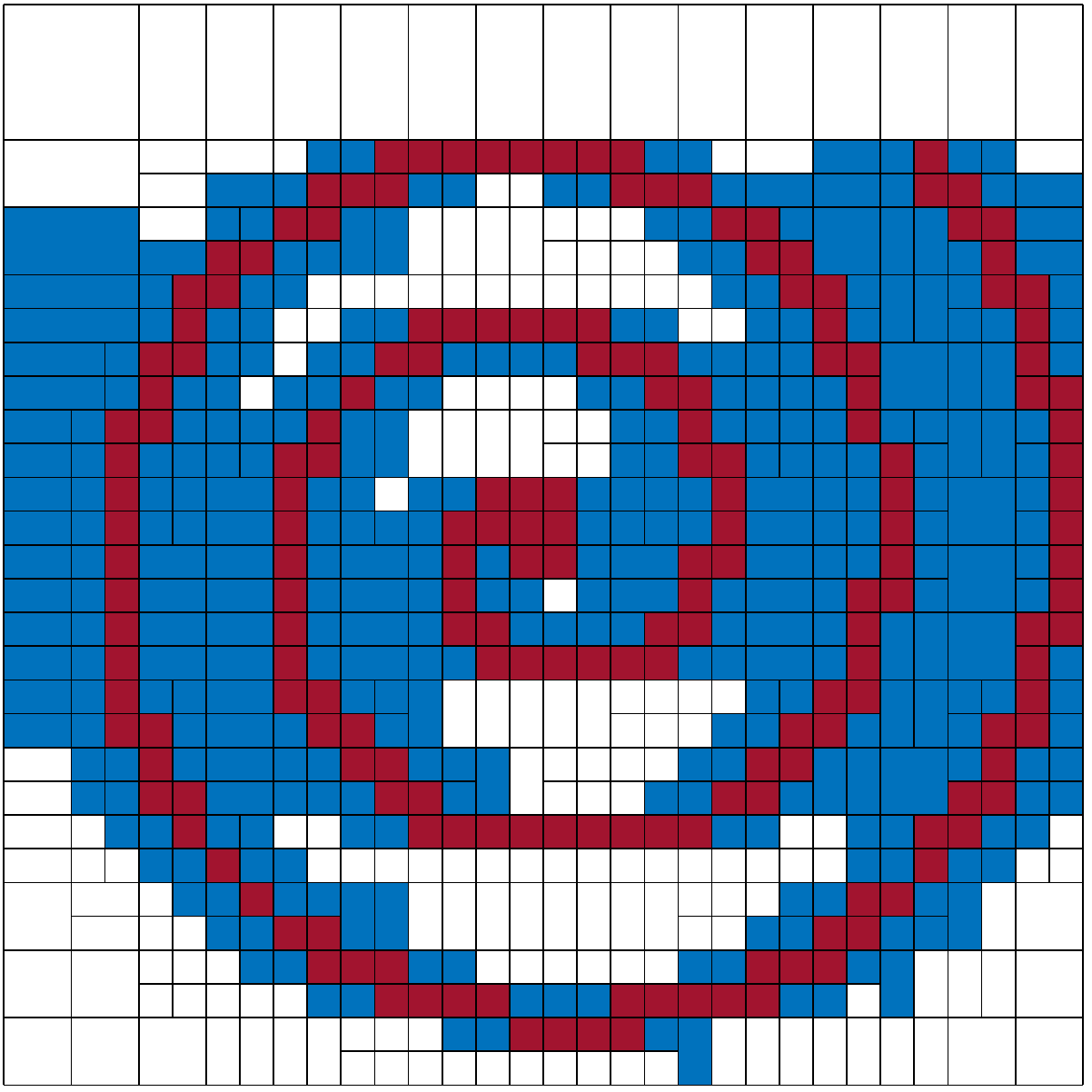}}\quad
\subfloat[]{
\includegraphics[width=.22\textwidth]{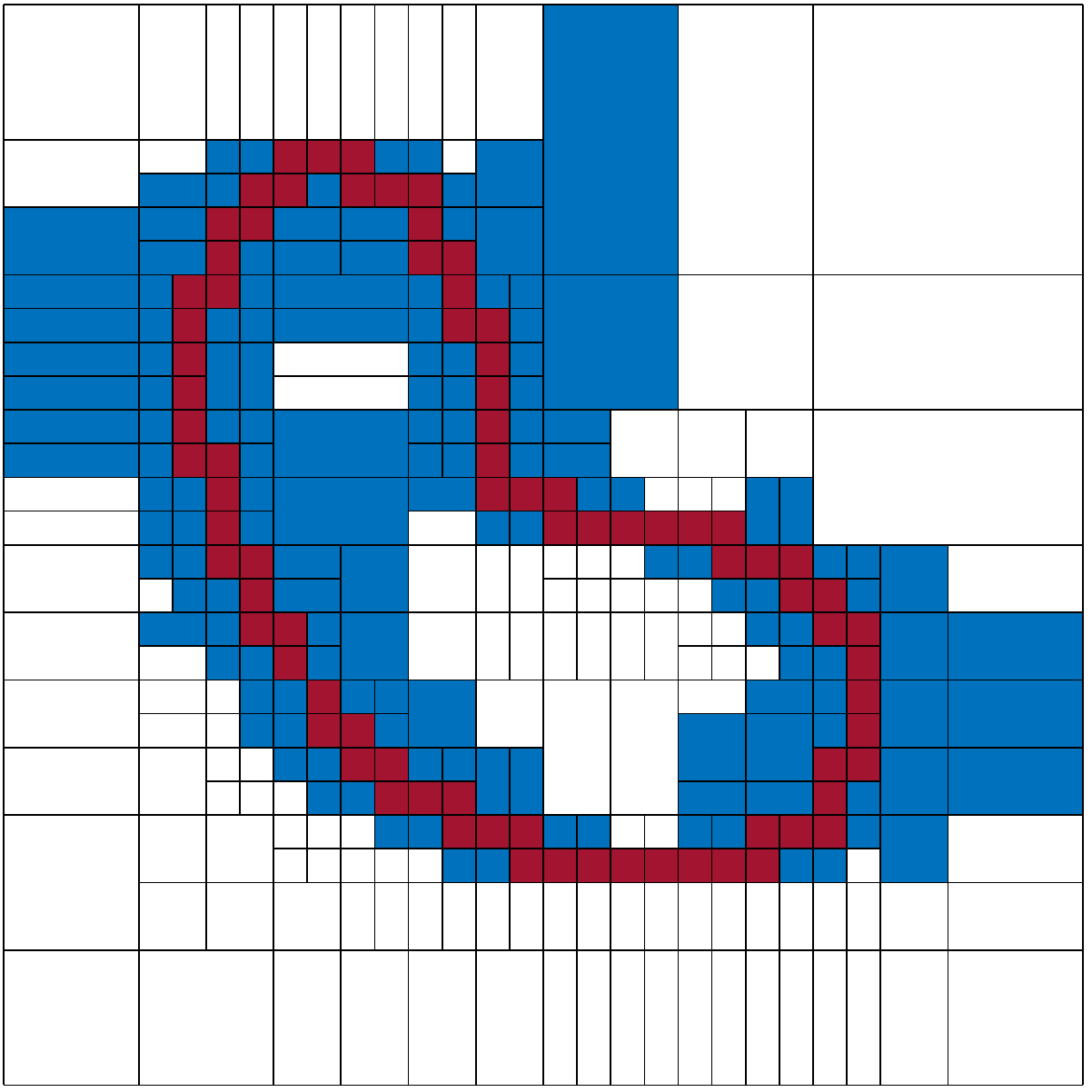}}
\caption{Examples of a horizontal generalized shadow map of different sets. All the meshlines in all the meshes have multiplicity one and $p_1 =2$. The red regions are the sets considered and the unions of the red regions and the blue regions are the shadow of them (we refer to the online version of the paper for the colors). In (a)--(b) the underlying mesh is a tensor mesh. In (c)--(d) we consider LR meshes built using the \emph{minimum span} strategy, proposed in \cite{johannessen}, with bidegree $(2,2)$ to emphasize the difference of the generalized shadow map on meshes with local insertions with respect to the tensor case.}\label{shadows}
\end{figure}
In Figure \ref{shadows} we show four examples of horizontal generalized shadow maps for three different sets and degree $p_1=2$. In particular the sets considered are unions of boxes of the underlying mesh. We made this choice because these are the kind of sets considered for refinement in practice.

In the EG strategy we will apply the generalized shadow map to sets composed of boxes of the same size and shape in the mesh. The direction of the shadow will be established by such shape: if the boxes are rectangles then the shadow is in the same direction of the long edges, if they are squares then it is in the other direction. 

\subsection{Definition of the strategy and proof of the \NS~property}\label{sec:EGdef}
Given a region $\omega \subseteq \Omega$ composed of a set of boxes to be refined, the EG strategy can be divided in two macro steps. In the first step new lines are inserted in order to refine $\omega$. As we shall show, these lines halve boxes of the same shape and size and therefore they are all in the same direction, as we explained in Section \ref{sec:EGpre}. The new line insertions will in general spoil the \NS~property of the mesh. In the second step of the EG strategy we reinstate the \NS~property by suitably extending lines that were already on the mesh before such new insertions. This approach, of dividing the strategy into ``refining step'' and ``\NS~property recovering step'', was already adopted in \cite{N2S2} for the \NNSS~ mesh refinement. As it will be clear, restoring the \NS~property will also provide nice grading properties in the final mesh.

The refining step works as follows. Let $\cN = (\cM, \pmb{p}, \mu)$ be the LR mesh at hand, provided by several iterations of EG strategy, and let $\cL$ be the corresponding set of LR B-splines. We define the subset $\cL_\omega \subseteq \cL$ as the set of those LR B-splines whose support is intersecting region $\omega$. Then we compute the maximum of $\diam(\beta)$ over all the LR B-splines $B \in \cL_\omega$ and all the boxes $\beta$ in the tensor meshes $\cM_B$ associated to the knot vectors of $B$. We halve such maximal boxes. As all of them have same diameter, the new lines have all same direction. This concludes the refining step. The new lines inserted and the new extensions, provided by the re-establishing of the \NS~property, trigger a refinement in the LR B-spline set $\cL$. We finally update $\omega$ by removing those boxes of it that have been refined (if any). We repeat the procedure until all the boxes in $\omega$ have been halved. The scheme of the EG refinement strategy is given in Algorithm \ref{alg:EG}.
\begin{algorithm}[h!]
\linespread{1.35}\selectfont
\Do{$\omega \neq \emptyset$}{
Set $\cL_\omega = \bigcup\{B \in \cL\,:\,\int(\supp B) \cap \omega \neq \emptyset\}$\;
Set $D = \max_{B \in \cL_\omega}\max_{\beta \in \cM_B} \diam(\beta)$\;
Update $\cM$ by halving the boxes $\beta \in \cM_B$ with $\diam(\beta) = D$ and $B \in \cL_\omega$\;
Reinstate the \NS~property and grading after the refinement, $\cN \leftarrow$ \texttt{EGgrader($\cN$)}\;
Update $\cL$\;
Update $\omega$ by removing the boxes of it that have been refined\;
}
\caption{EG strategy iteration, $(\cL,\cN) \leftarrow$ \texttt{EGstrategy($\cL$,$\cN$,$\omega$)}}\label{alg:EG}
\end{algorithm}

When we recover the \NS~and grading properties we make sure that the shadow of each box of diameter $d$ in the mesh contains only boxes of diameter $sd$ or smaller, with the scaling value $s$ defined as in Section \ref{sec:EGpre}. We proceed from the boxes with the smallest diameter to those with the largest diameter. The input is the LR mesh $\cN = (\cM, \pmb{p}, \mu)$ obtained after the refining step. Let $\cE$ be the set of boxes of $\cM$. At first, we set $d$ as the diameter of the smallest boxes in $\cM$ and $\cE_d \subseteq \cE$ as the set of boxes with diameter $d$. For each of such boxes $\beta \in \cE_d$ we check if there is a $\beta' \in \cS \beta$ with $\diam(\beta') > sd$. If this is the case, we halve the closest to $\beta$ of such larger boxes and we update the shadow of $\beta$. We iterate this procedure until all the boxes in $\cS \beta$ have diameter at most $sd$. After that, the next extensions will involve only boxes of diameter $sd$ or larger. Hence, we remove $\cE_d$ from $\cE$, we update $d$ as the smallest diameter of the boxes in such new collection. We iterate the procedure until $\cE$ becomes empty. 

The \NS~property restoring step is schematized in Algorithm \ref{alg:EGgrader}.
\begin{algorithm}[h!]
\linespread{1.35}\selectfont
Set $\cE$ as the set of boxes in $\cM$\;
\While{$\cE \neq \emptyset$}{
Set $d = \min_{\beta \in \cE} \diam(\beta)$ and $\cE_d = \{\beta\in \cE \,:\, \diam(\beta) = d\}$\;
\For{all $\beta \in \cE_d$}{
\While{$\exists\, \beta' \in \cS \beta$ with $\diam(\beta') > sd$}{
Update $\cM$ by halving the closest of such $\beta'$ to $\beta$\;
Update $\cE$\;
Update $\cS \beta$\;
}}
Update $\cE$ by removing $\cE_d$, $\cE \leftarrow \cE \setminus \cE_d$\;
}
\caption{Restoring \NS~and grading properties, $\cN \leftarrow$ \texttt{EGgrader($\cN$)}}\label{alg:EGgrader}
\end{algorithm}
In Figure \ref{fig:exEGiteration} we visually represent the steps of an iteration of the EG refinement on a given LR mesh.
\begin{figure}
\centering
\subfloat[]{\includegraphics[width=.3\textwidth]{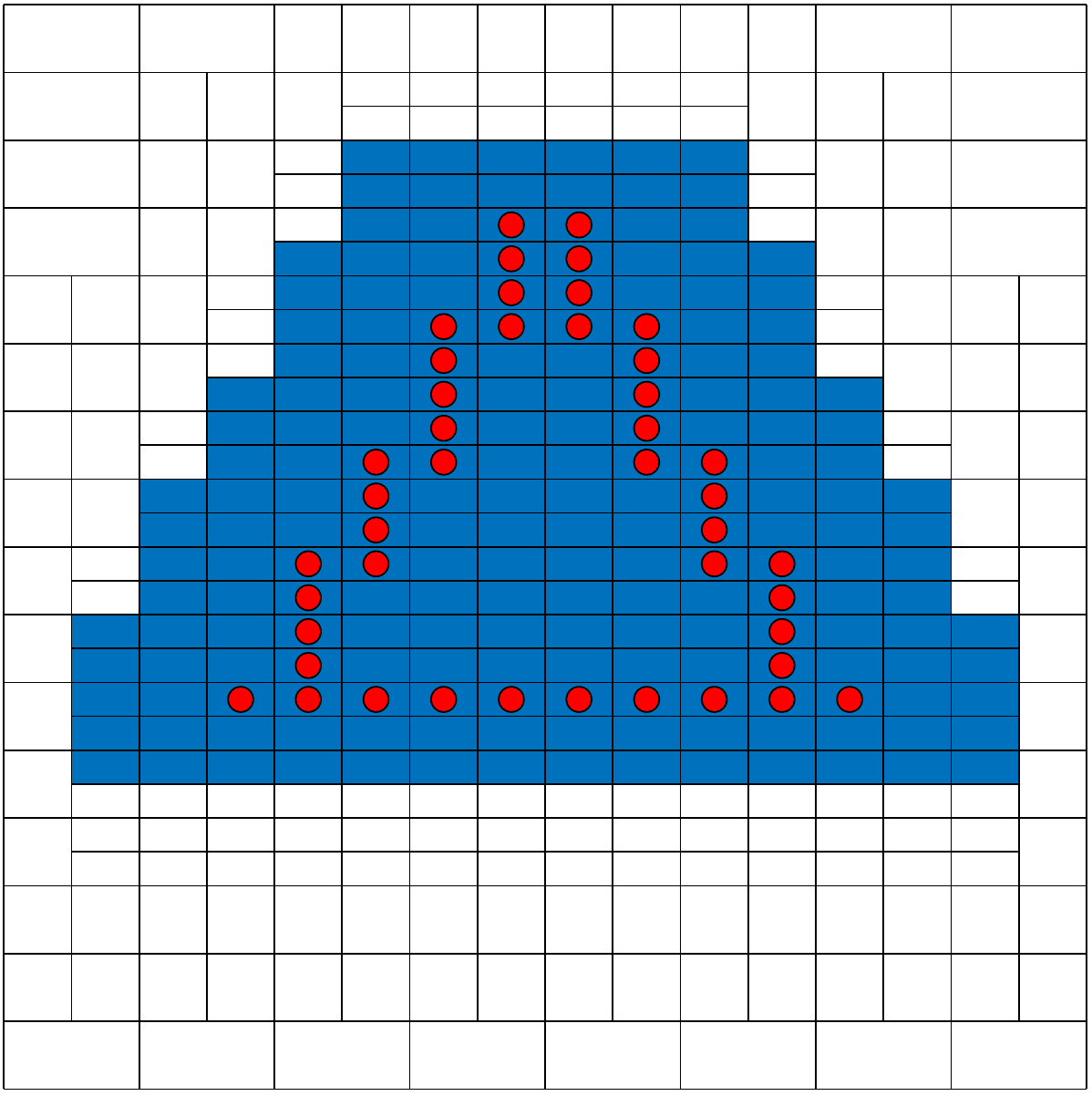}}\quad
\subfloat[]{\includegraphics[width=.3\textwidth]{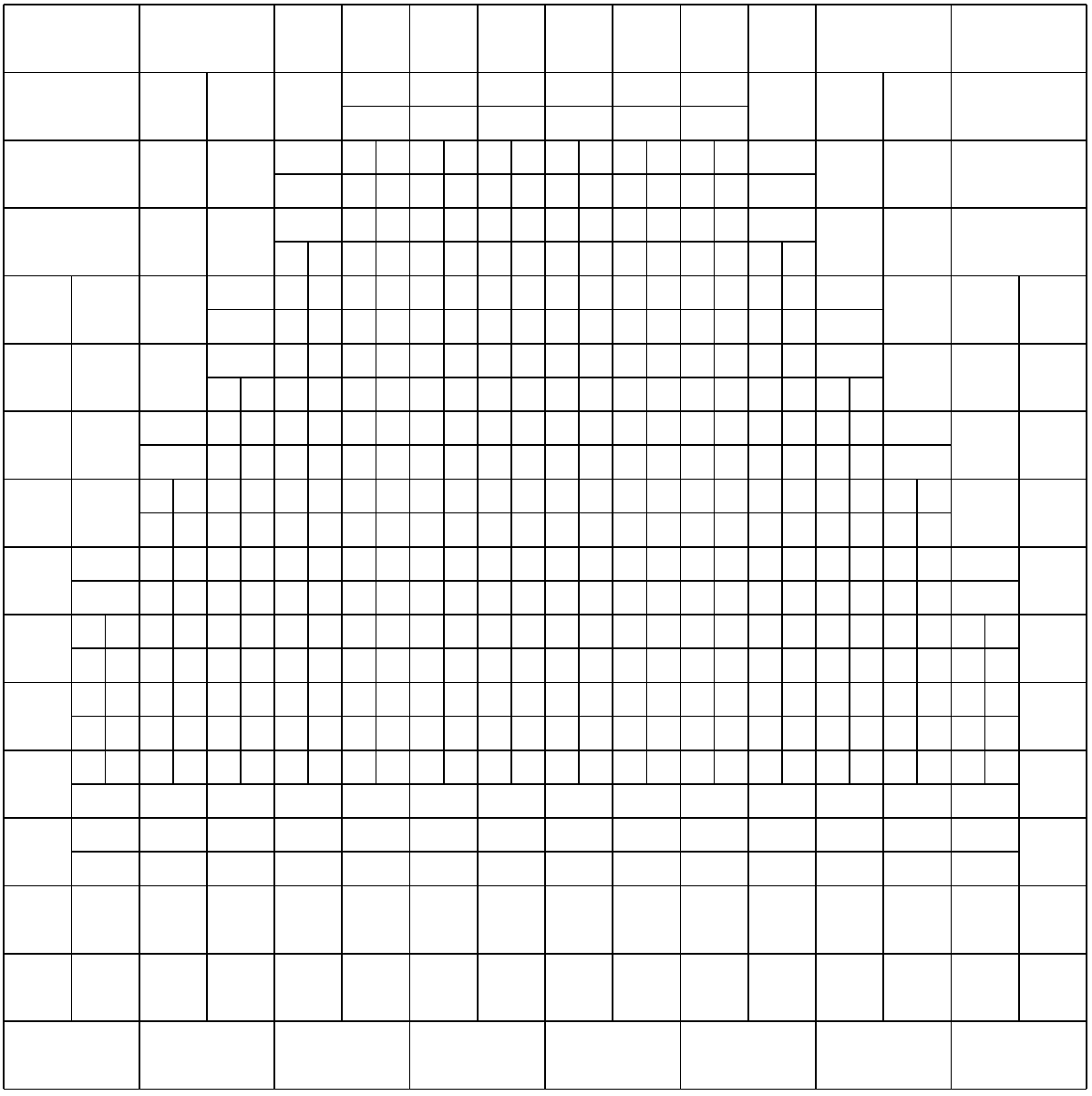}}\quad
\subfloat[]{\includegraphics[width=.3\textwidth]{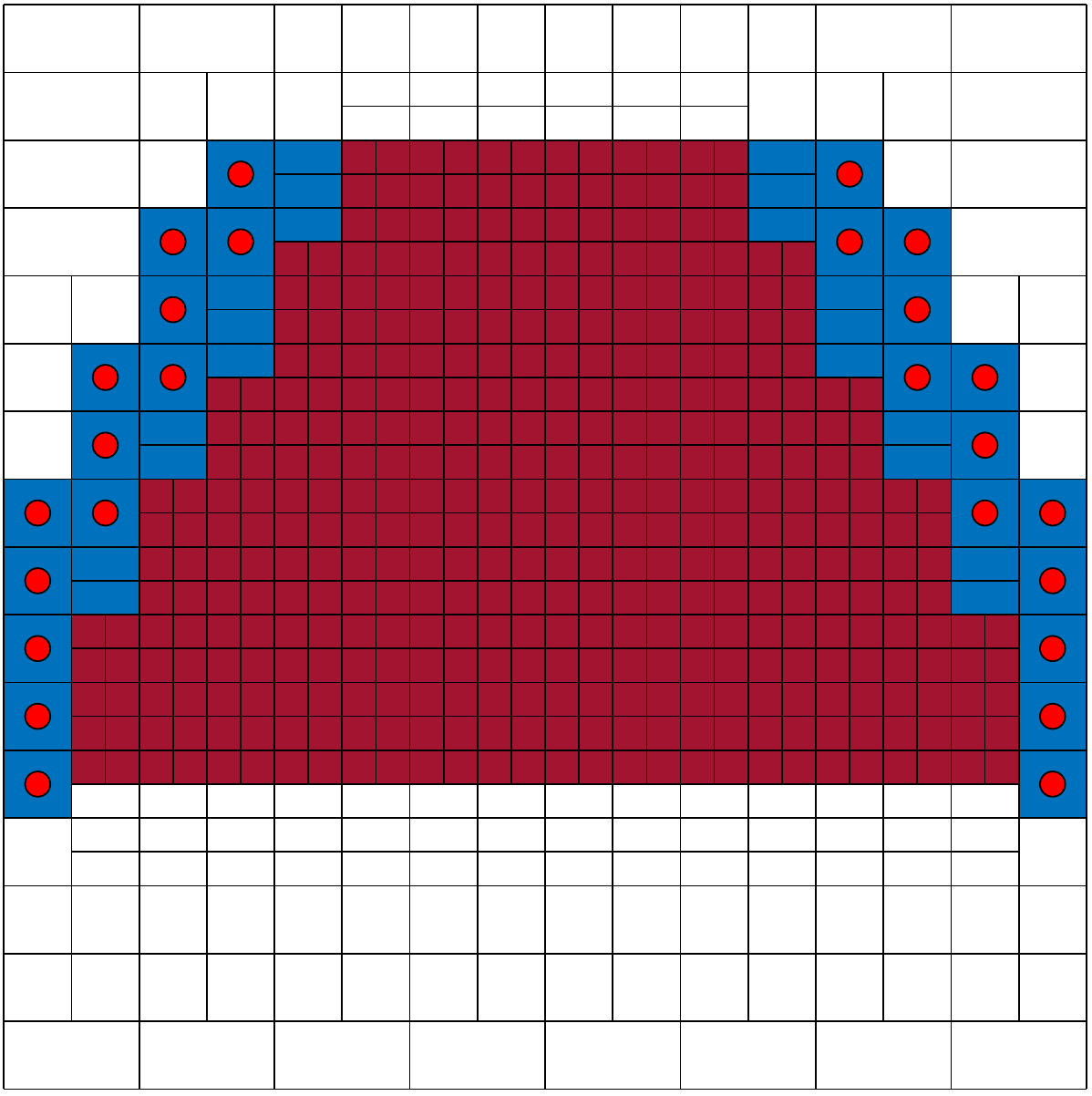}}\\
\subfloat[]{\includegraphics[width=.3\textwidth]{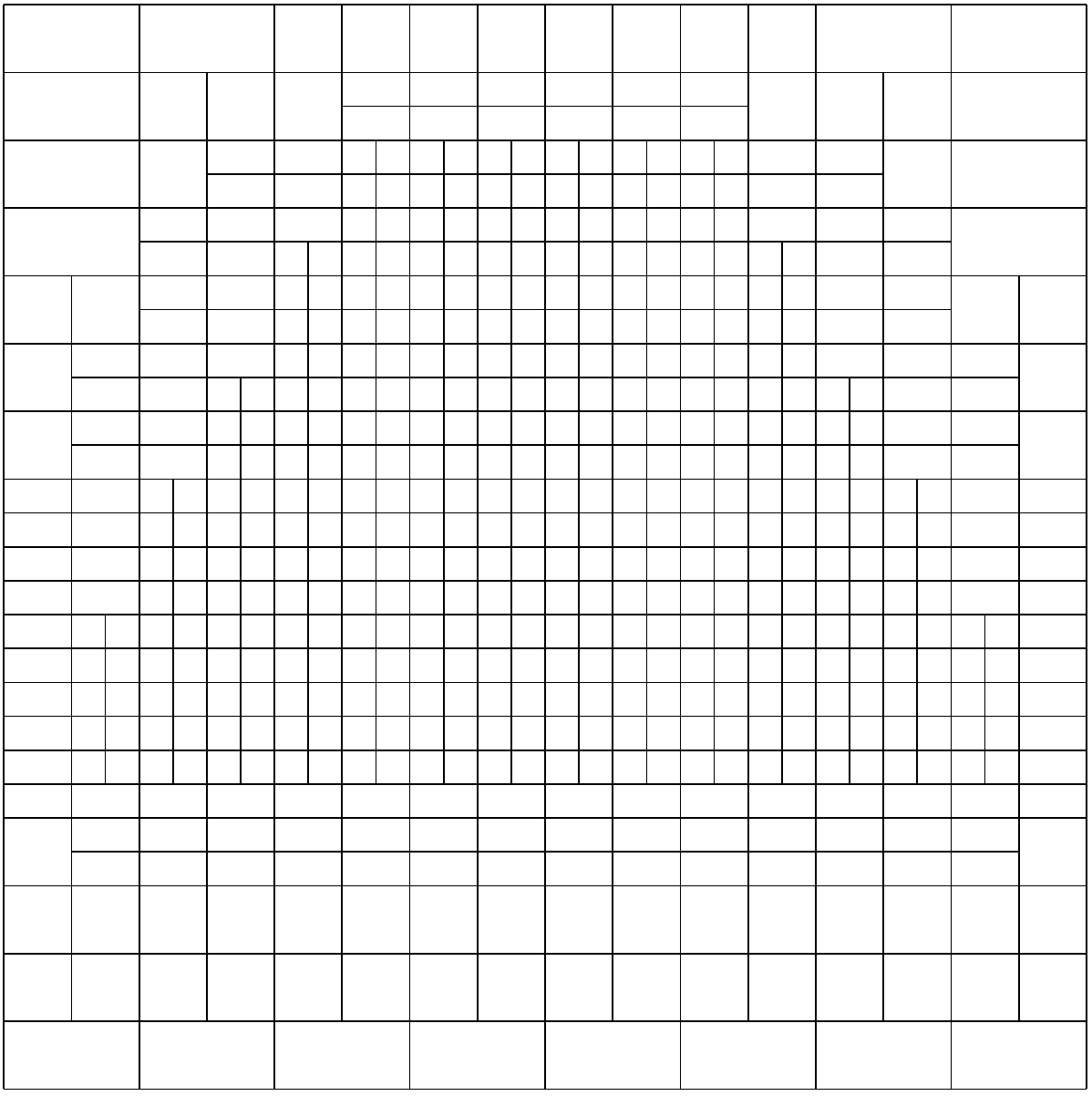}}\quad
\subfloat[]{\includegraphics[width=.3\textwidth]{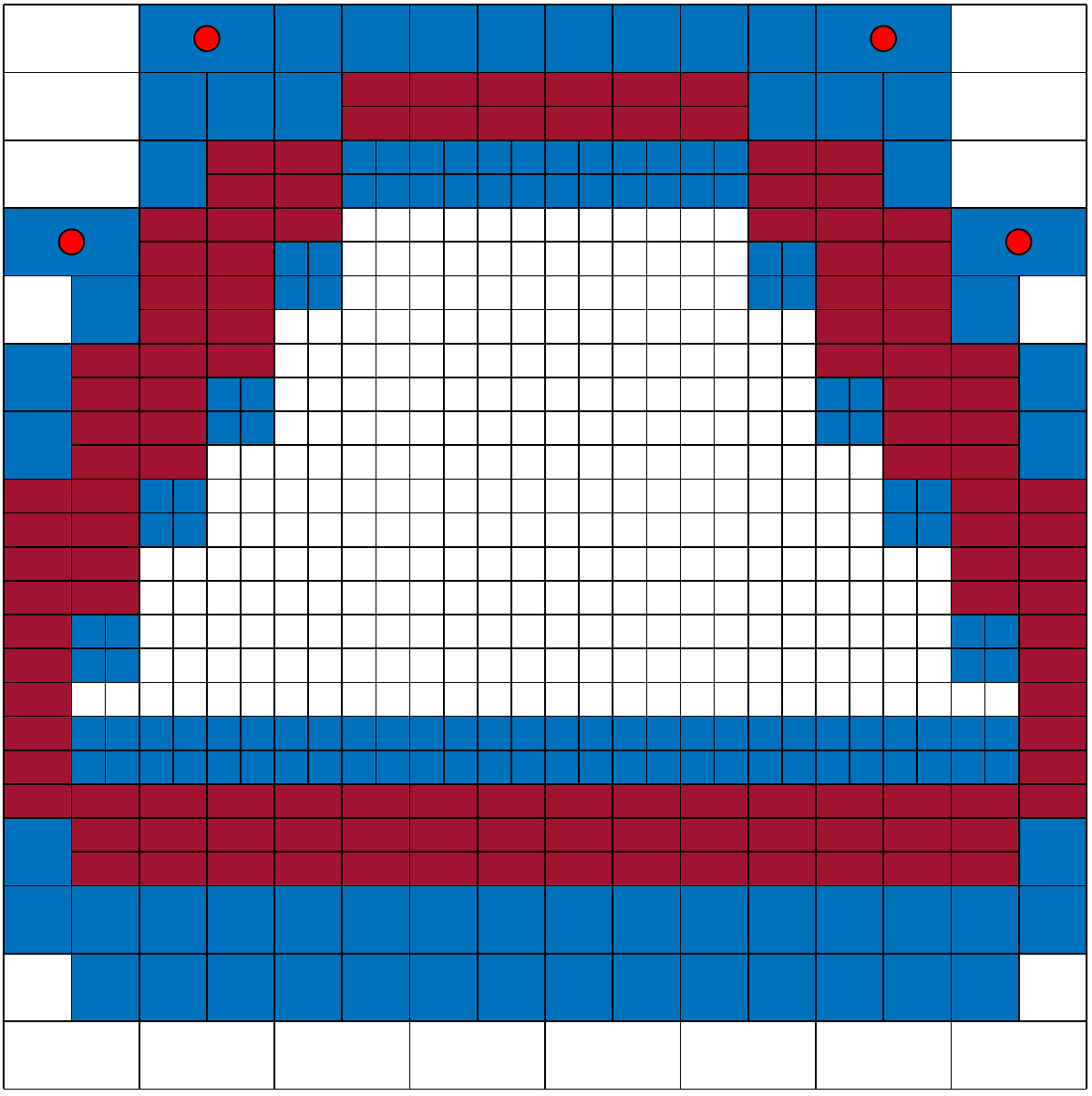}}\quad
\subfloat[]{\includegraphics[width=.3\textwidth]{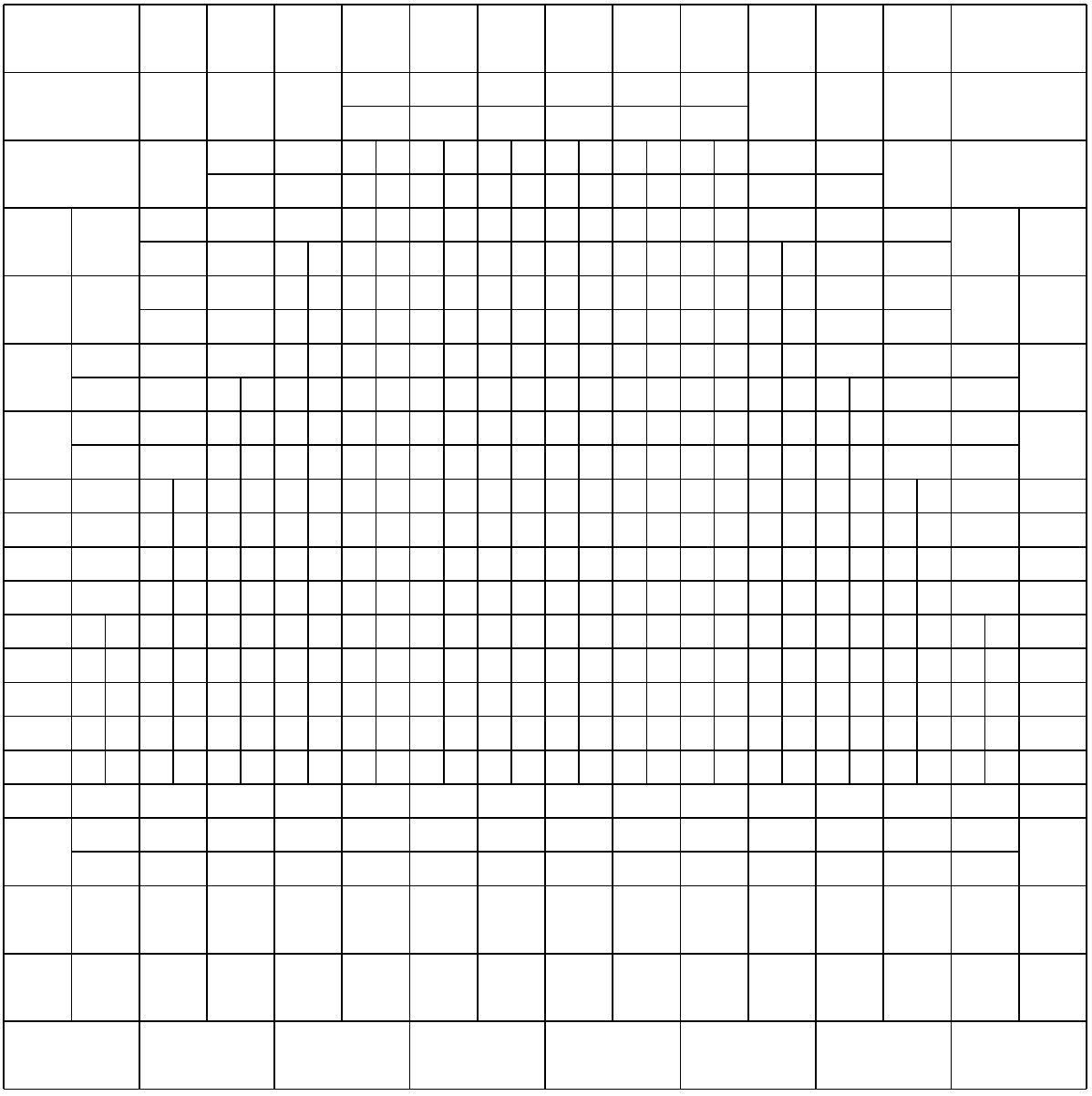}}
\caption{Example of EG strategy iteration. The input is the LR mesh pictured in figure (a) and a set of boxes marked for refinement, highlighted with dots. We collect all the LR B-splines on the mesh whose support intersects the marked boxes. The region given by the union of their supports is colored in figure (a). We halve the boxes of largest diameter in their support, in this case all the boxes in the colored region, and we get the LR mesh shown in figure (b). Such LR mesh may have not the \NS~property. We reinstate it as follows. We consider the smallest boxes on the mesh and we compute the generalized shadow of such region. We consider the Horizontal-major variant of the strategy for this example, therefore such shadow is horizontal, as shown in figure (c). We mark for refinement those boxes in the shadow that are too large as explained in Algorithm \ref{alg:EGgrader}. These boxes are highlighted with a dot in figure (c). We halve those of them that are closer to the region and we update the shadow. At the end of the process we have the mesh in figure (d). We iterate the procedure over all the boxes from the smaller to the larger. The next boxes considered are those reported in figure (e) together with their (vertical) generalized shadow. The boxes marked with a dot are those that need to be halved. The final LR mesh is reported in figure (f).}\label{fig:exEGiteration}
\end{figure}
We remark that the LR meshes produced by the EG strategy have boundary meshlines of full multiplicity and internal meshlines of multiplicity 1. 

In order to prove that such LR meshes have the \NS~property, we rely on the following result \cite[Theorem 11]{bressan2}. Let $\{\cN_\ell^T = (\cM^\ell, \pmb{p}, \mu)\}_{\ell\in\NN}$ be a sequence of tensor meshes with $\cM_0^T$ the boundary of $\Omega$ and $\cM_\ell^T$ obtained by halving the boxes in $\cM_{\ell-1}^T$, alternating the directions of such splits. Let $\Omega_\ell\subseteq \Omega$ be a union of boxes in $\cM_\ell^T$. Then \cite[Theorem 11]{bressan2} states that if an LR mesh $\cN =(\cM,\pmb{p},\mu)$ can be written as $\cM = \cup_{\ell \leq L} \cM_{\ell}^T|_{\Omega_\ell}$ and the sequence $\{\Omega_\ell\}_{\ell \leq L}$ is such that $\Omega_{\ell-1}\supseteq \cS \Omega_\ell$, then $\cN$ has the \NS~property. 
We now show that the LR meshes produced by the EG strategy satisfy the hypotheses of \cite[Theorem 11]{bressan2}.

\begin{thm}\label{N2Sproperty}
Let $\cN = (\cM,\pmb{p},\mu)$ be an LR mesh obtained via several iterations of the EG strategy. Then $\cN$ has the \NS~property.
\end{thm}
\begin{proof}
Let $d$ be the minimal diameter over all the boxes of $\cM$. Let $\Omega^d\subseteq \Omega$ be the region composed of all the boxes in $\cN$ of diameter $d$. Let $d' = sd$ and $\Omega^{d'}\subseteq \Omega$ be the region made of boxes of diameter $d'$ or smaller. In the \NS~property restoring step of the EG strategy (Algorithm \ref{alg:EGgrader}) we make sure that only boxes of diameter $sd$ or smaller are in $\cS\Omega^d$. Therefore, $\Omega^{d'} \supseteq \cS \Omega^d$. By iterating this procedure, replacing $d$ with $d'$ until $\Omega^{d'} = \cS\Omega^d = \Omega$, we get a sequence $\{\Omega^d\}_d$ for which $\Omega^{d'} \supseteq \cS \Omega^d$. Furthermore,
by recalling that the boxes of diameter $d$ are obtained by halving boxes of diamter $d'$, it is clear that the sequence $\{\Omega^d\}_d$ corresponds to a sequence $\{\Omega_\ell\}_{\ell\leq L}$ as that considered in \cite[Theorem 11]{bressan2} and $\cM = \cup_{\ell \leq L} \cM_\ell^T|_{\Omega_\ell}$. This proves that $\cN$ has the \NS~property thanks to \cite[Theorem 11]{bressan2}.
\end{proof}

In Figures \ref{fig:EGex1}--\ref{fig:EGex2} we show iterations of the EG strategy and the adaptivity of it. From LR meshes obtained by performing 14 iterations (7 vertical and 7 horizontal insertions) of the EG strategy localized on some regions, we change completely the curve along which we perform further refinements. All the meshes shown (and many more) have been tested for the \NS~property to confirm the theoretical result of Theorem \ref{N2Sproperty}.
\begin{figure}
\begin{tikzpicture}
\matrix (m)[matrix of math nodes,column sep=2em,row sep=2em]{
\includegraphics[width=.3\textwidth]{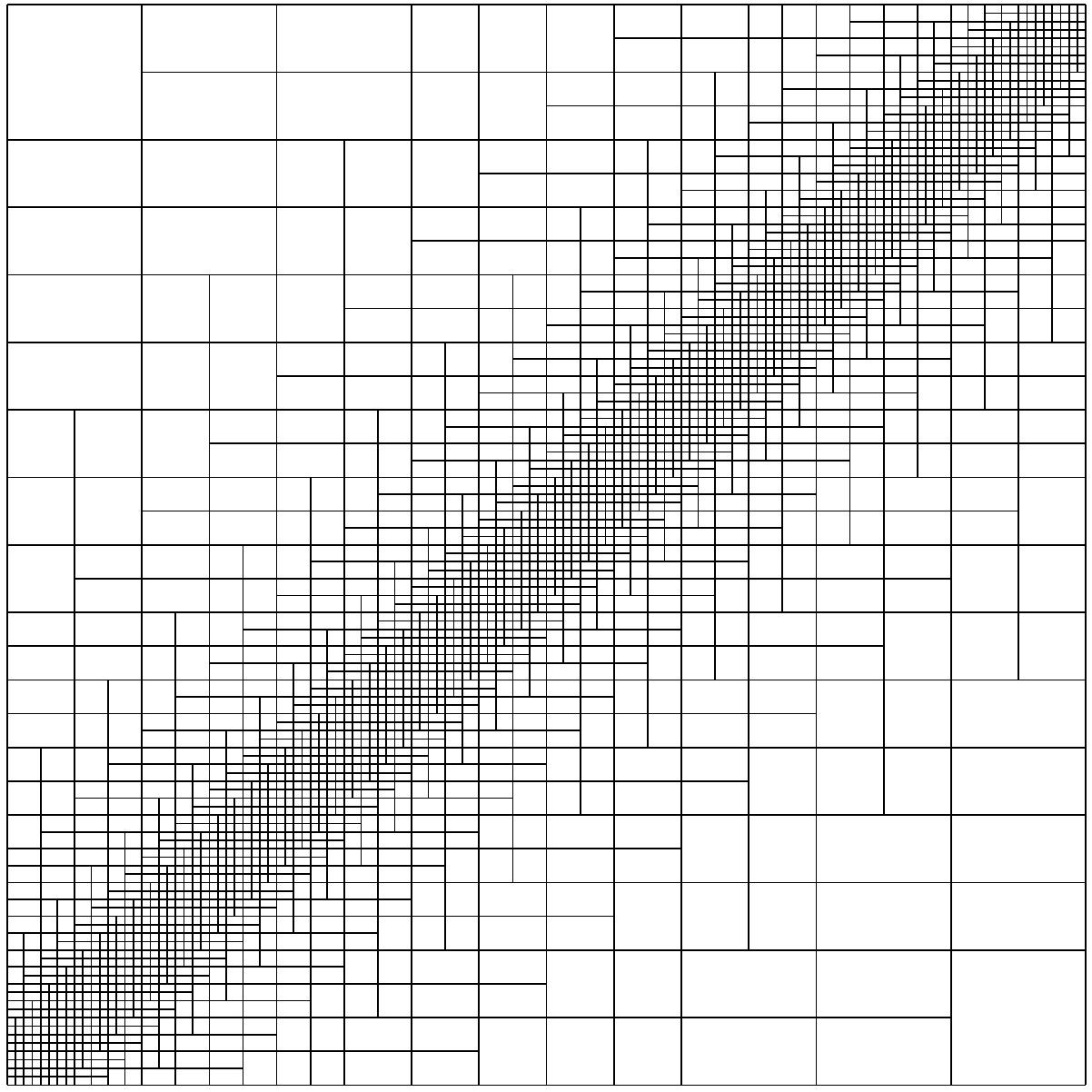}\pgfmatrixnextcell\includegraphics[width=.3\textwidth]{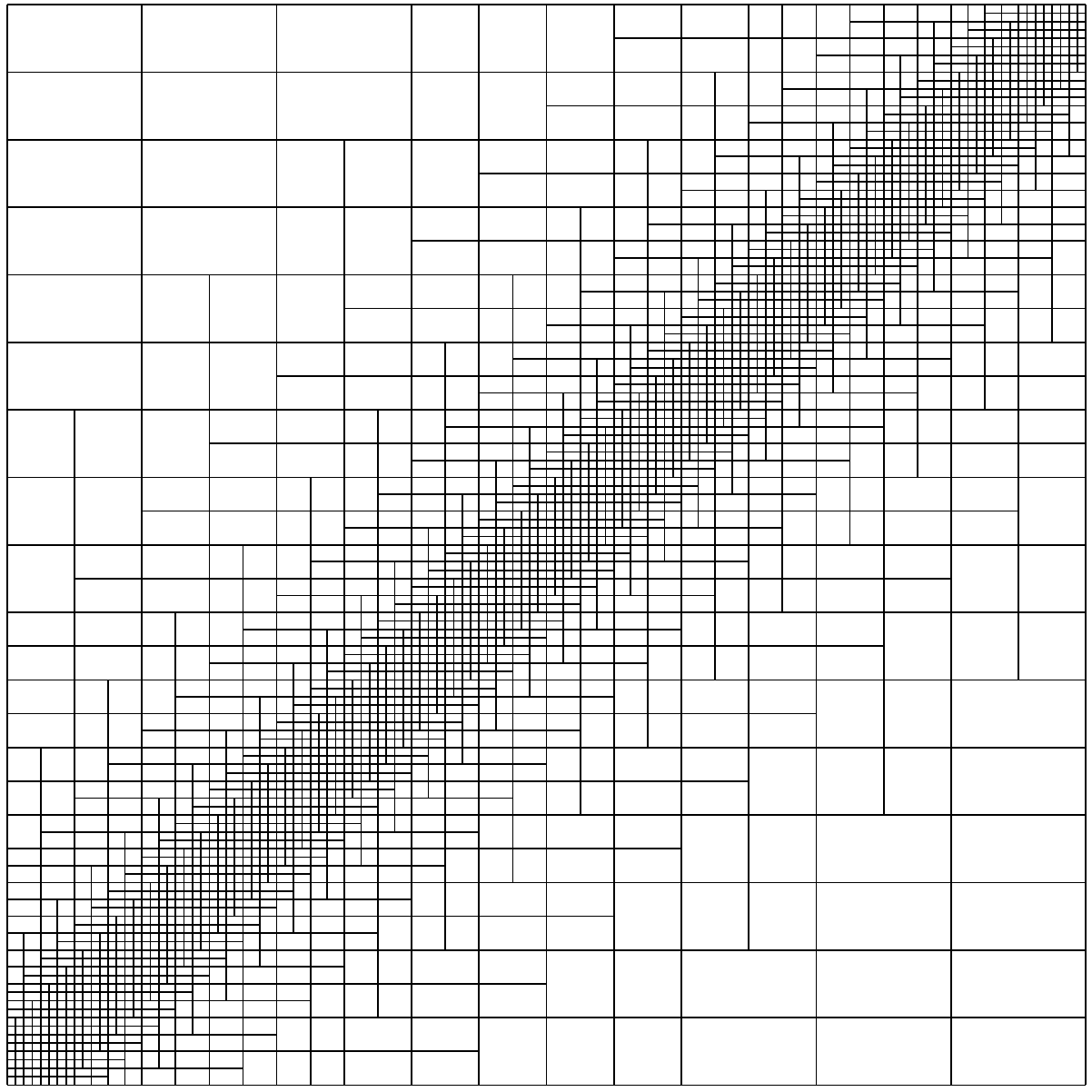}\pgfmatrixnextcell\includegraphics[width=.3\textwidth]{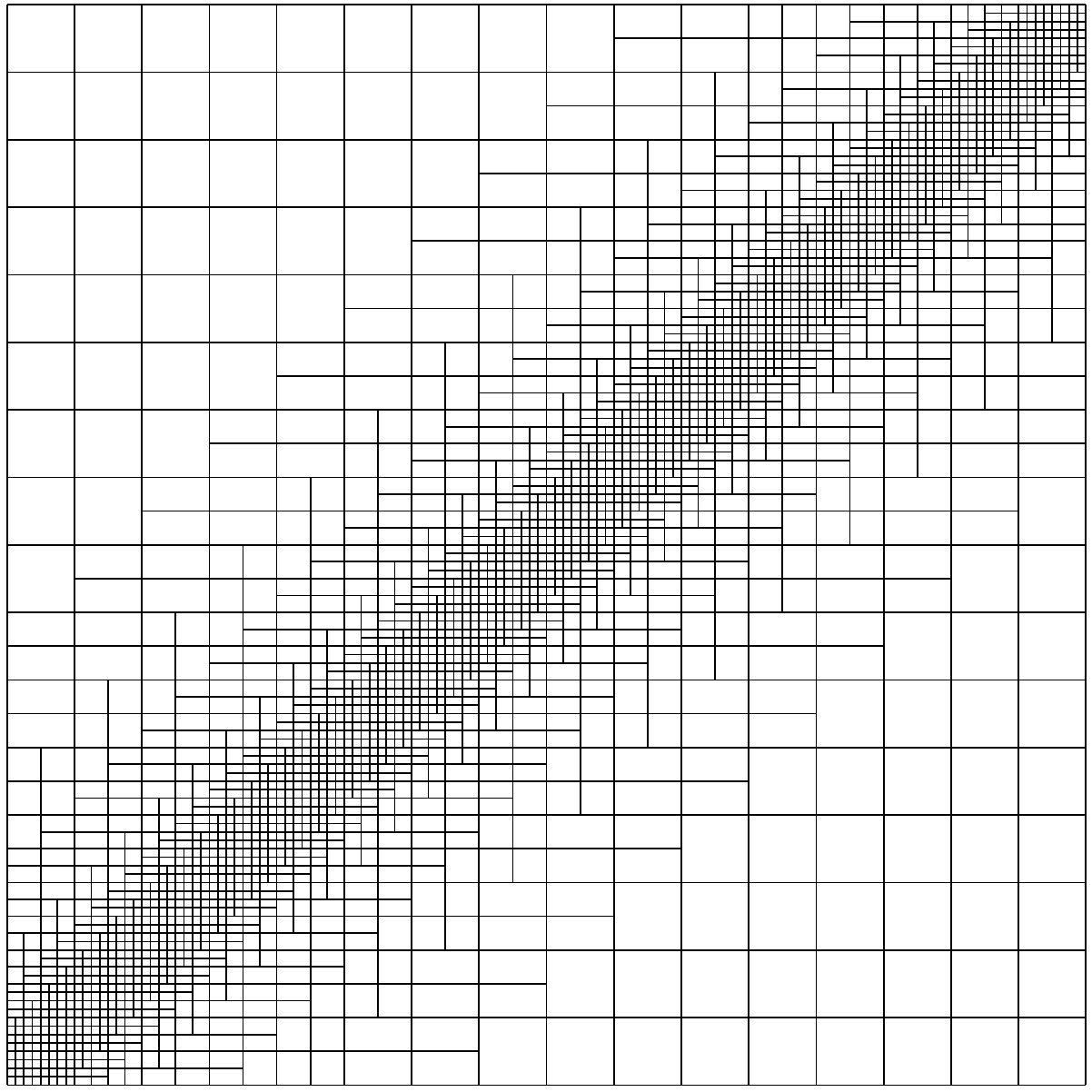}\\
\includegraphics[width=.3\textwidth]{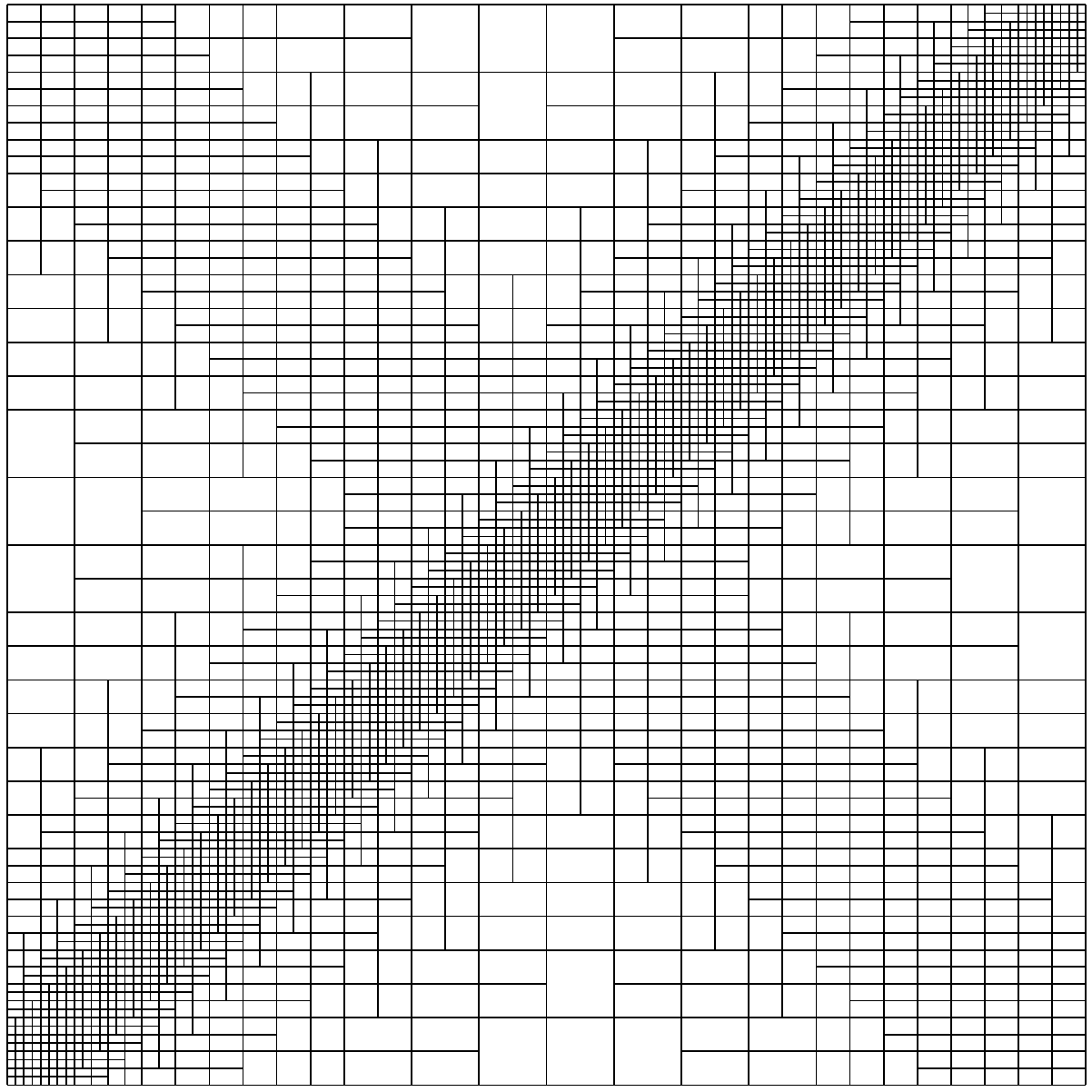}\pgfmatrixnextcell\includegraphics[width=.3\textwidth]{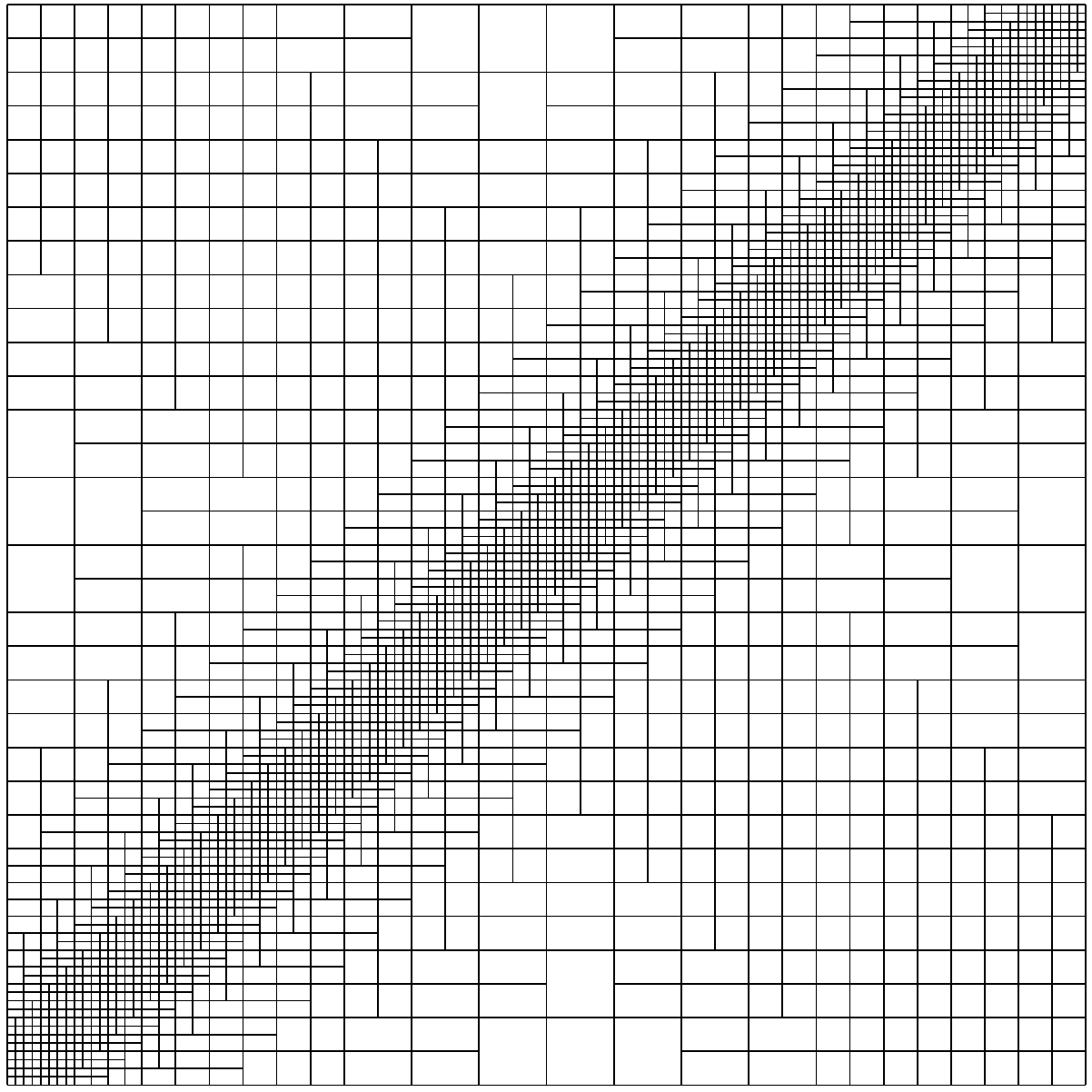}\pgfmatrixnextcell\includegraphics[width=.3\textwidth]{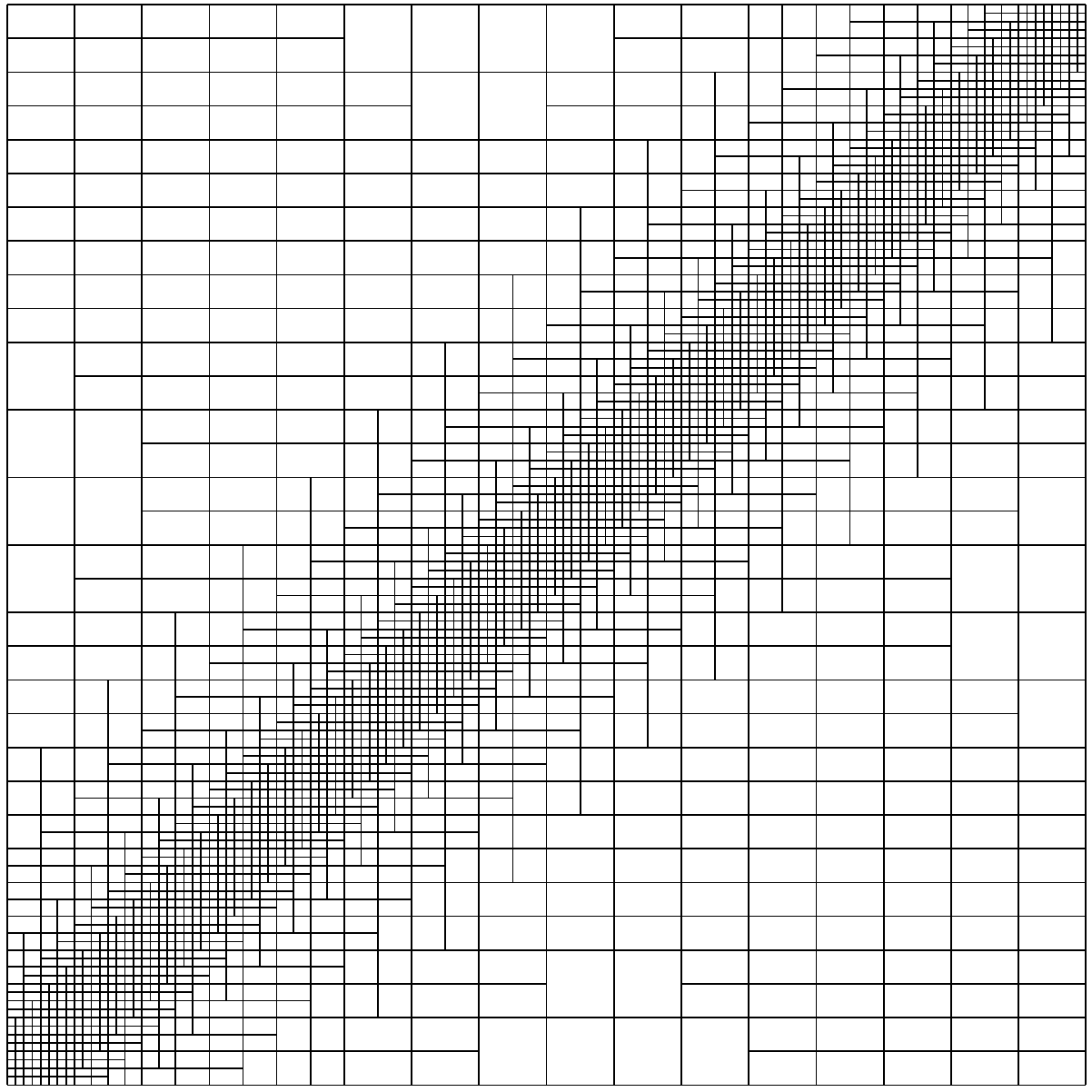}\\
\includegraphics[width=.3\textwidth]{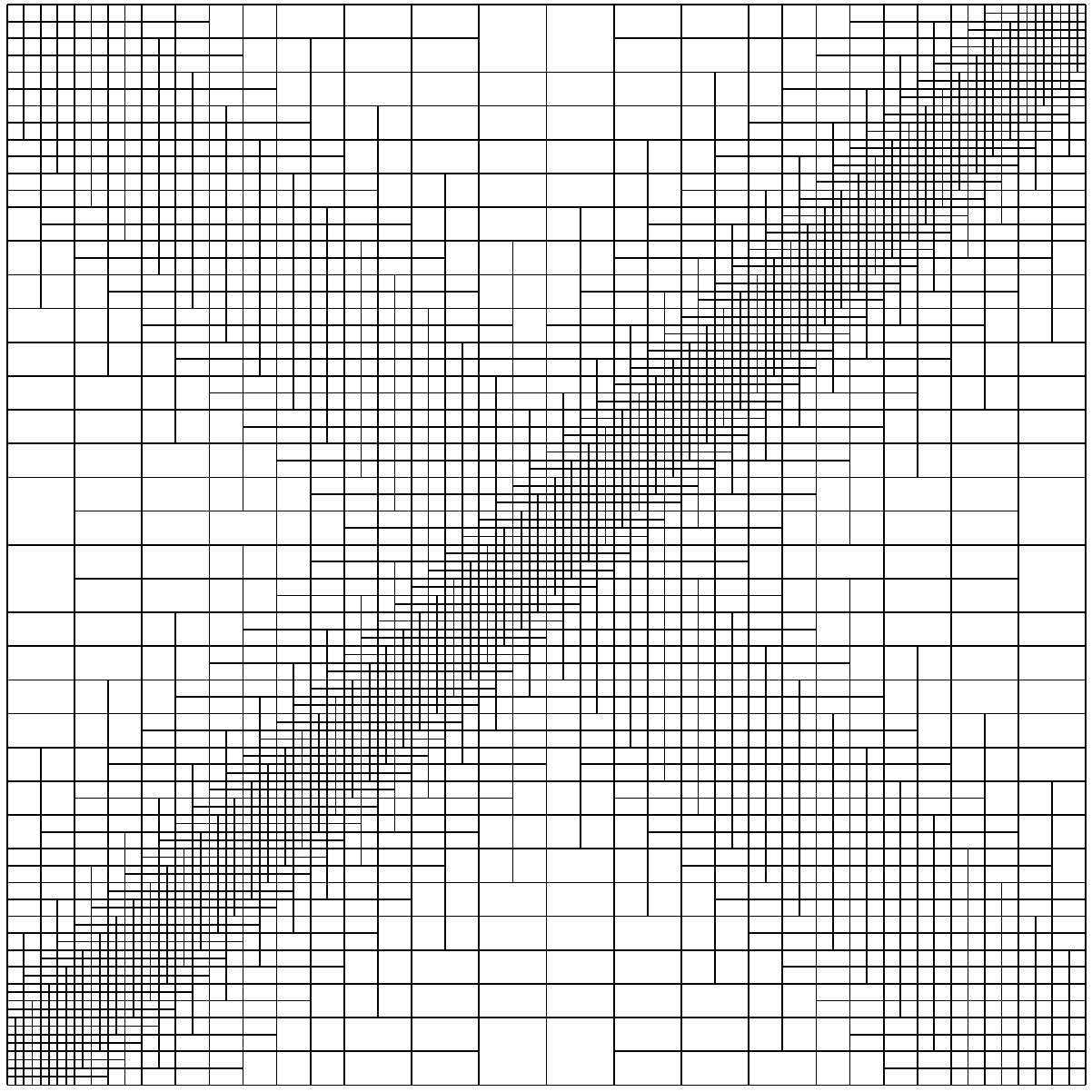}\pgfmatrixnextcell\includegraphics[width=.3\textwidth]{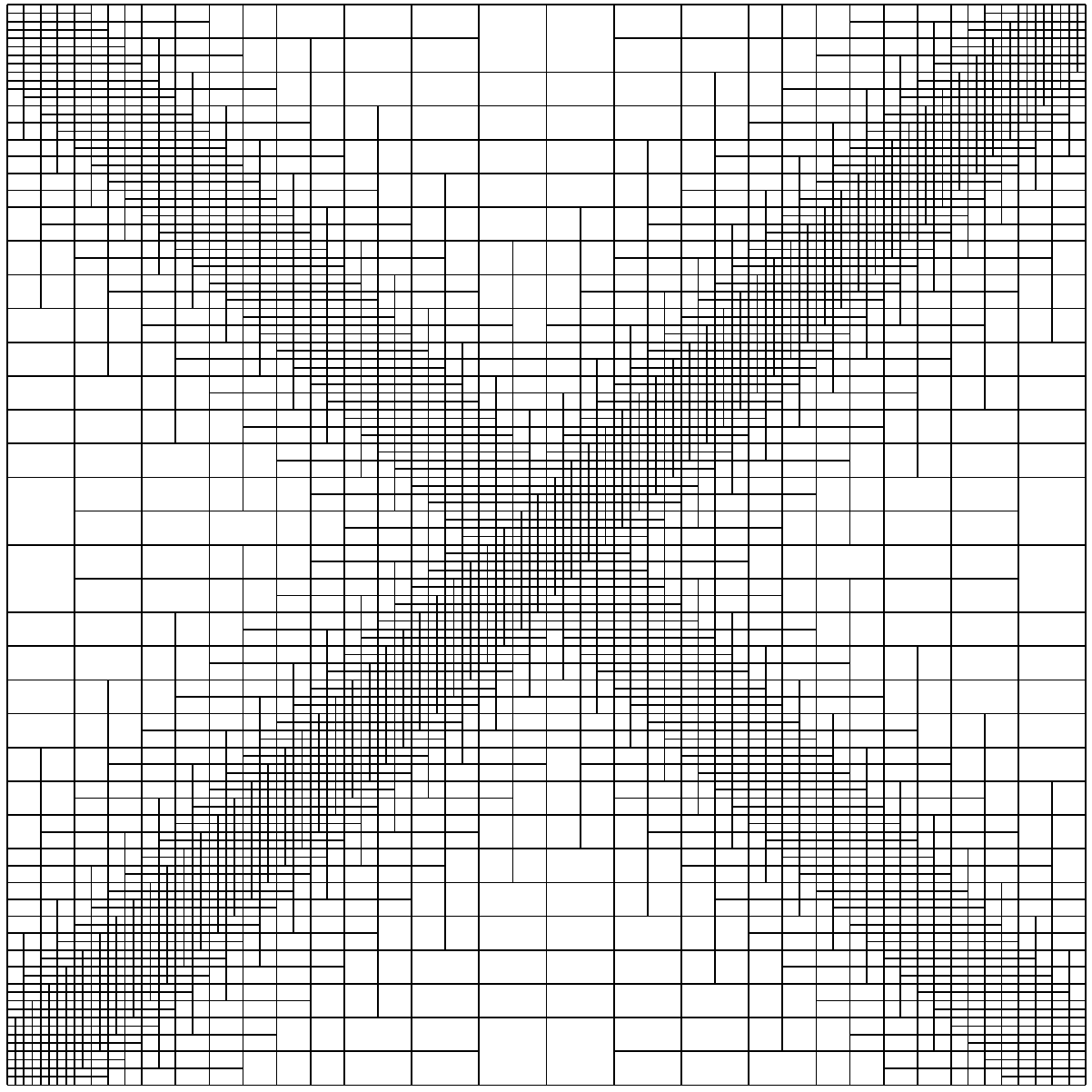}\pgfmatrixnextcell\includegraphics[width=.3\textwidth]{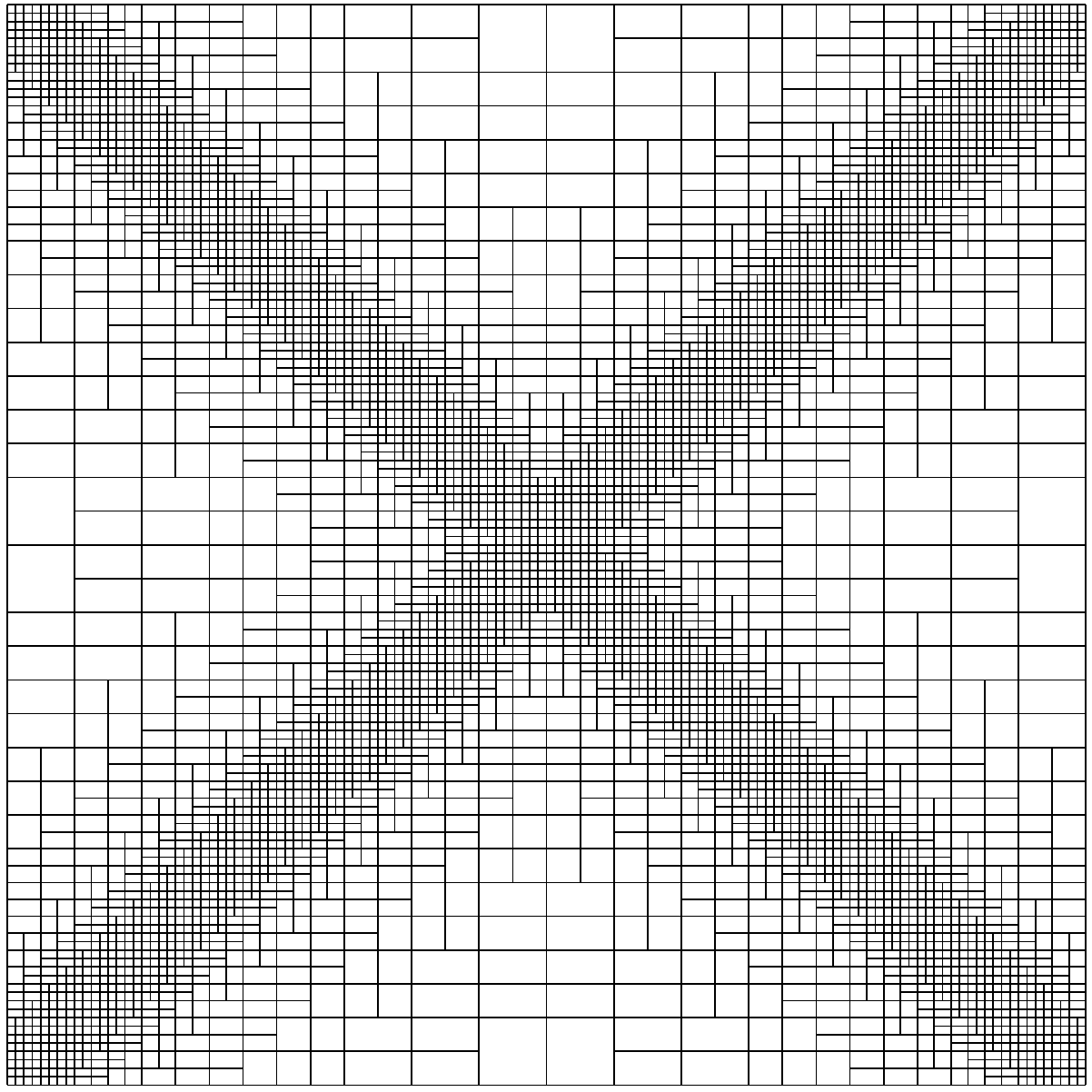}\\
};

\draw[\freccia] (m-1-1) -- (m-1-2);
\draw[\freccia] (m-1-2) -- (m-1-3);
\draw[\freccia] (m-1-3) -- (m-2-3);
\draw[\freccia] (m-2-3) -- (m-2-2);
\draw[\freccia] (m-2-2) -- (m-2-1);
\draw[\freccia] (m-2-1) -- (m-3-1);
\draw[\freccia] (m-3-1) -- (m-3-2);
\draw[\freccia] (m-3-2) -- (m-3-3);
\end{tikzpicture}
\caption{Example showing the adaptivity of the EG strategy. From a refinement localized along a diagonal, we perform iterations on the other diagonal to form an ``X'', switching the region of refinement. The figure has to be read following the arrows which represent the iterations. The EG strategy guarantees local linear independence of the LR B-splines on each of the LR meshes in the process. The bidegree considered is $\pmb{p}=(2,2)$.}\label{fig:EGex1}
\end{figure}

\begin{figure}
\begin{tikzpicture}
\matrix (m)[matrix of math nodes,column sep=2em,row sep=2em]{
\includegraphics[width=.3\textwidth]{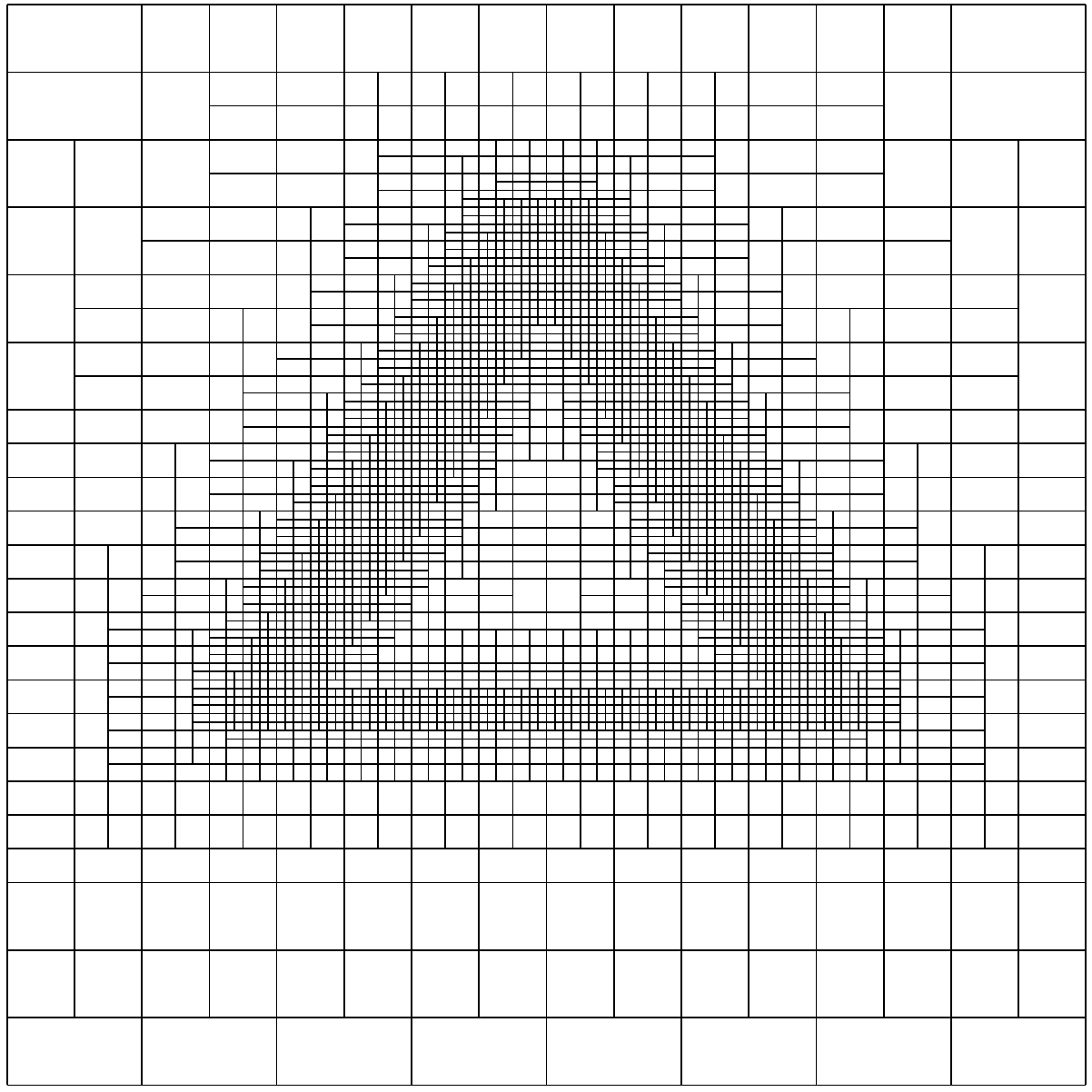}\pgfmatrixnextcell\includegraphics[width=.3\textwidth]{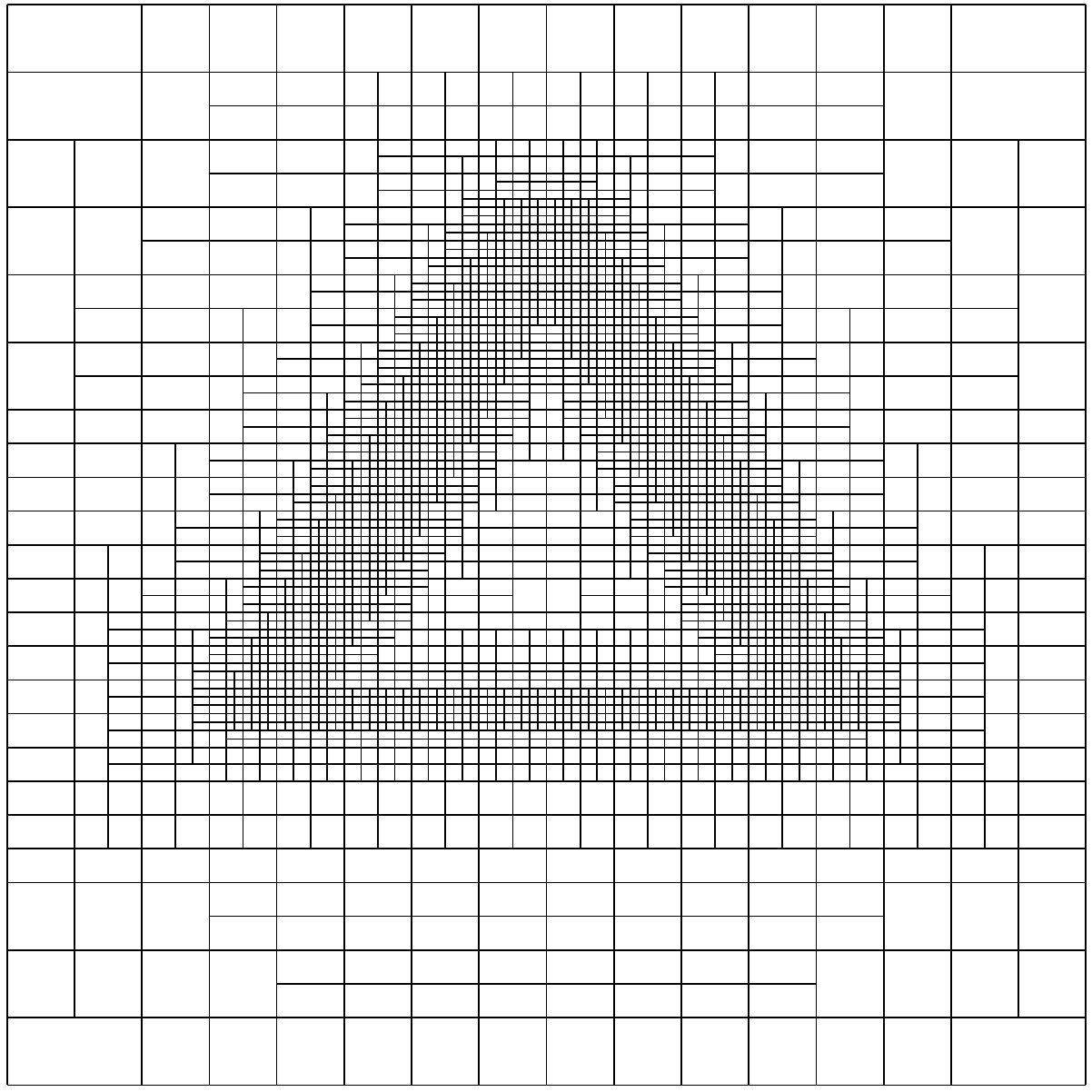}\pgfmatrixnextcell\includegraphics[width=.3\textwidth]{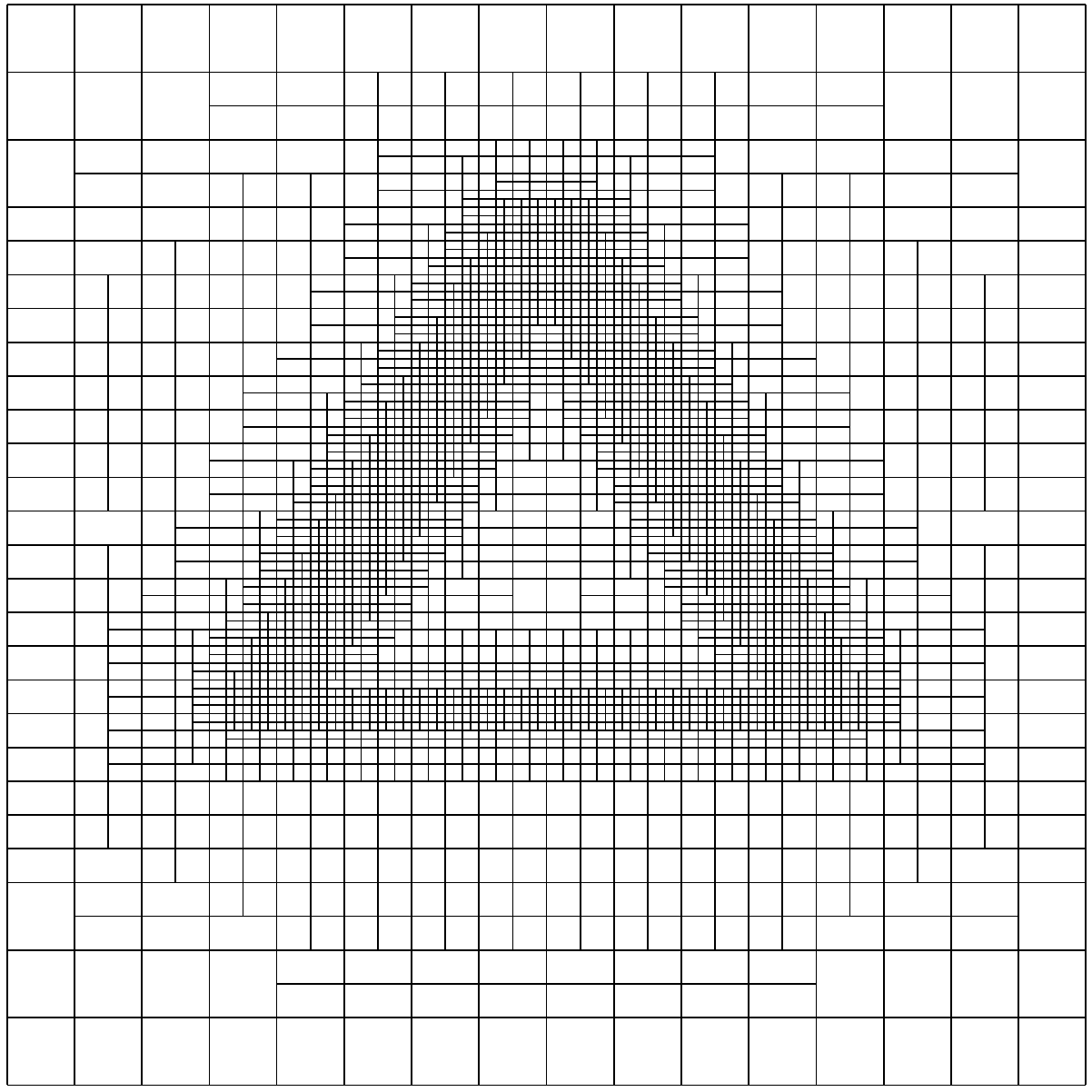}\\
\includegraphics[width=.3\textwidth]{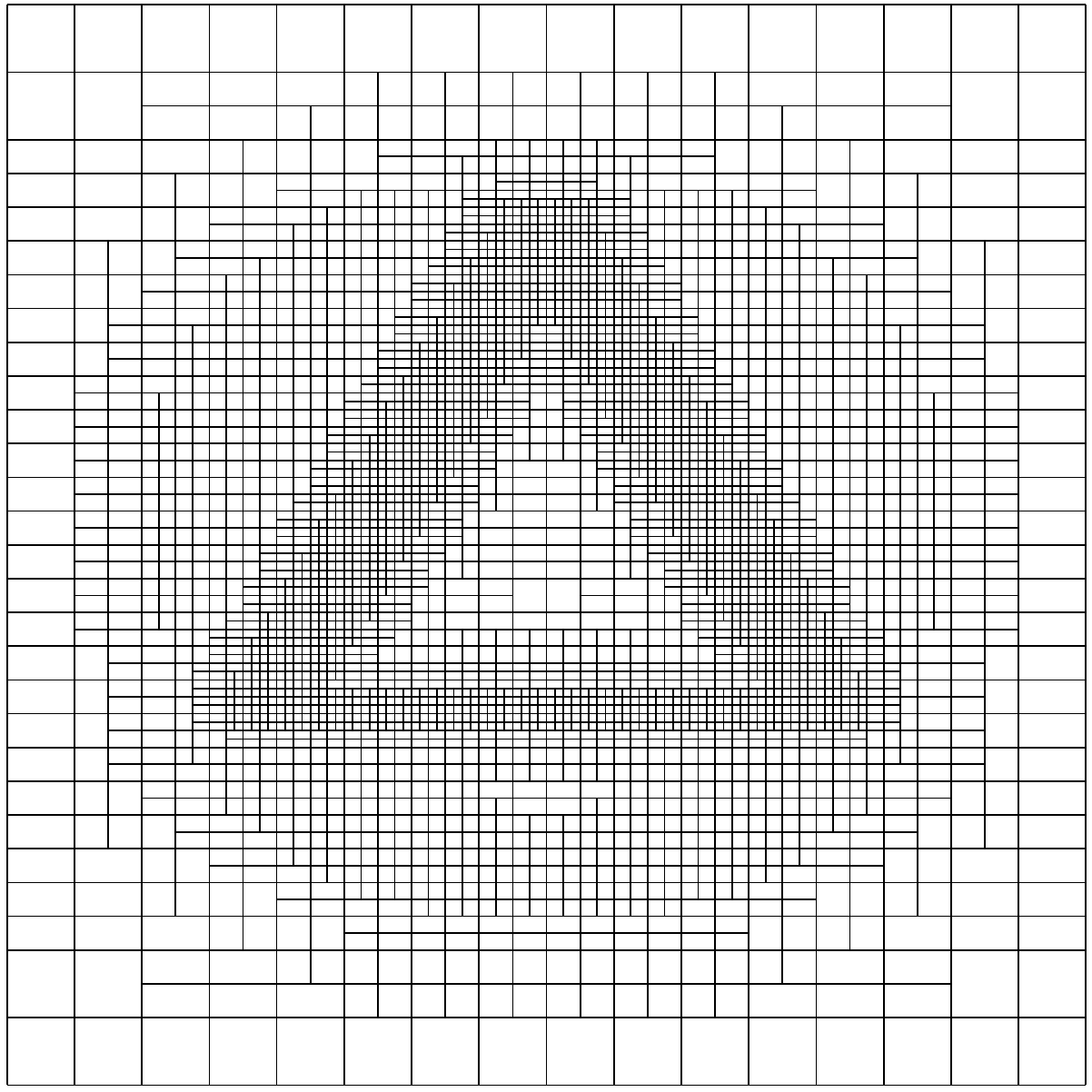}\pgfmatrixnextcell\includegraphics[width=.3\textwidth]{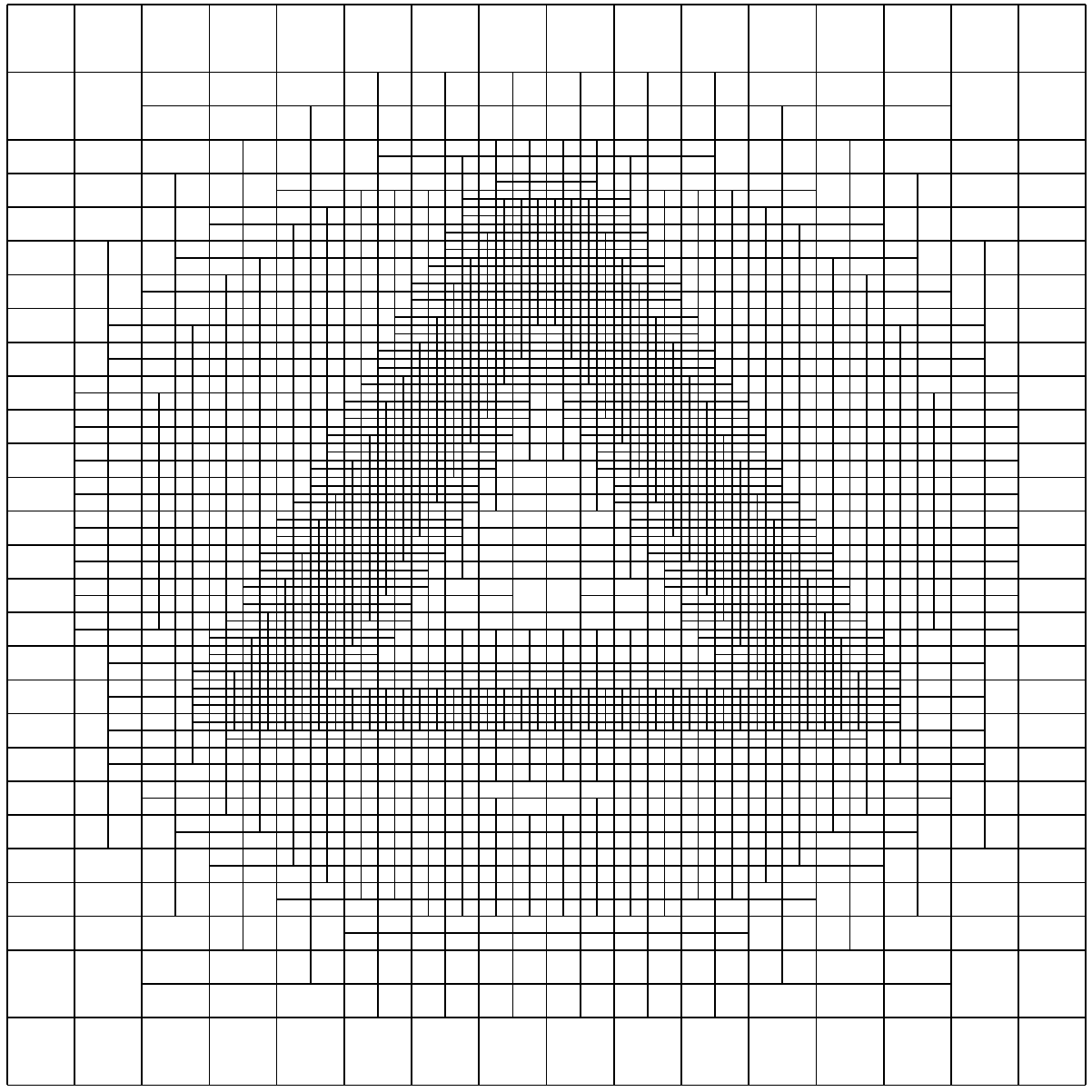}\pgfmatrixnextcell\includegraphics[width=.3\textwidth]{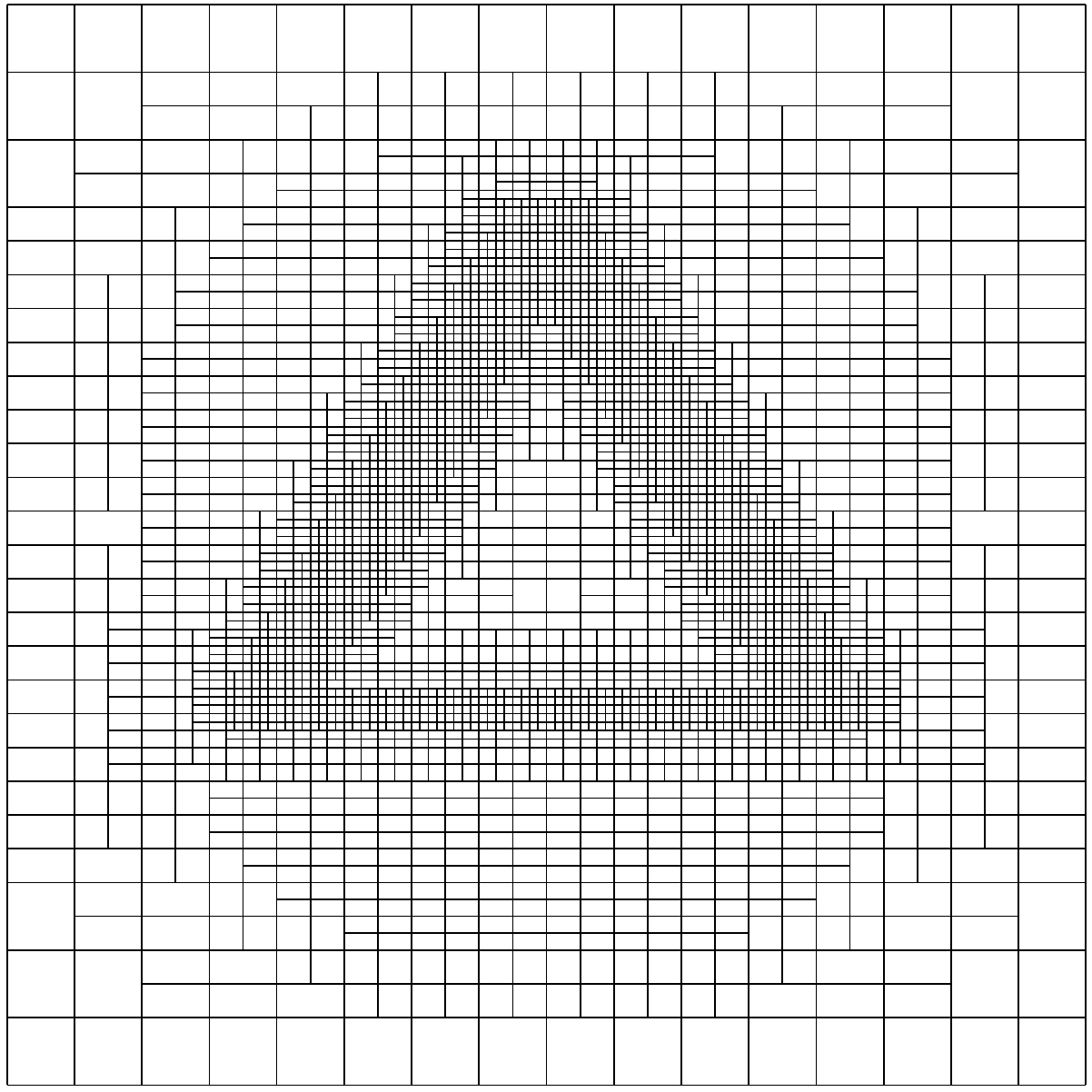}\\
\includegraphics[width=.3\textwidth]{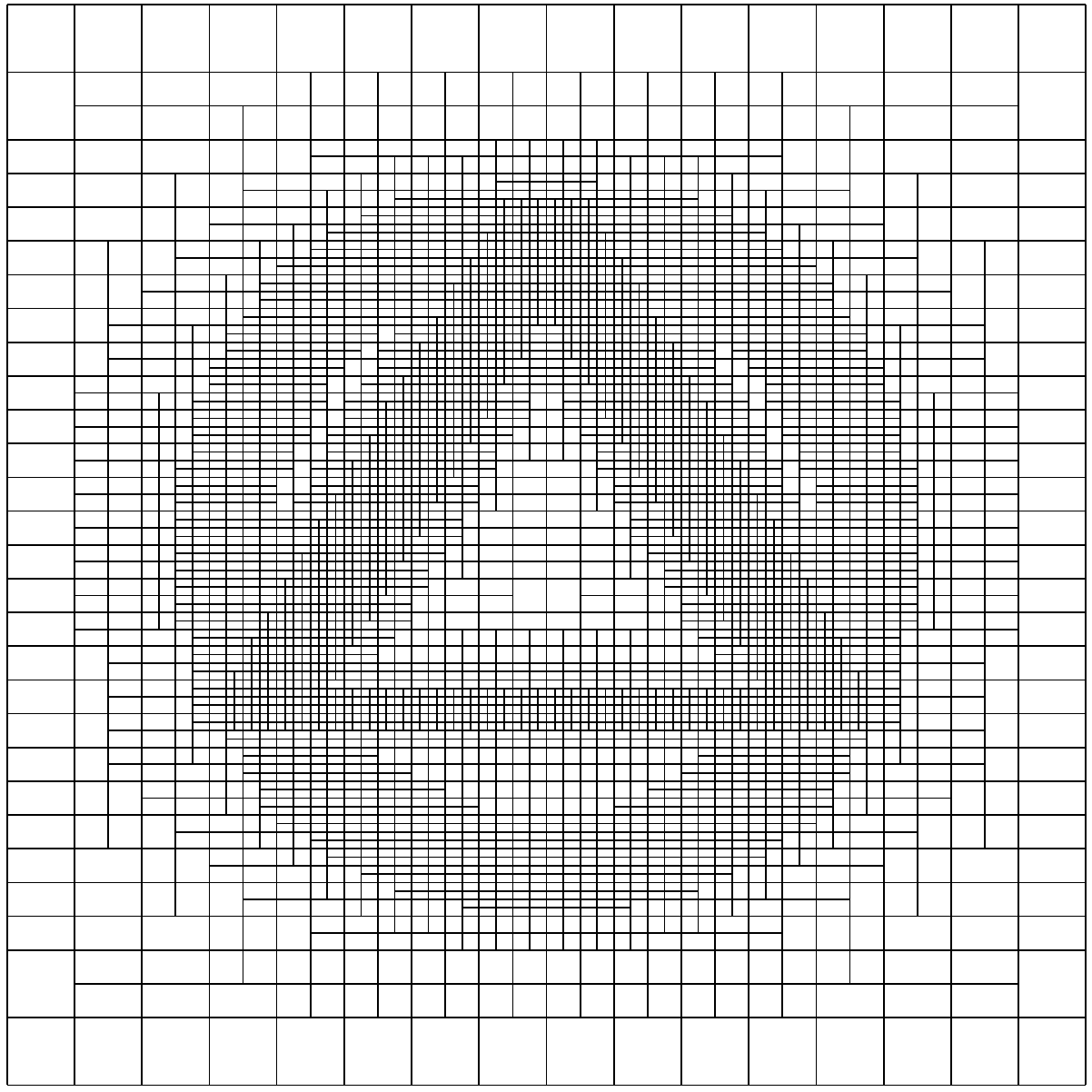}\pgfmatrixnextcell\includegraphics[width=.3\textwidth]{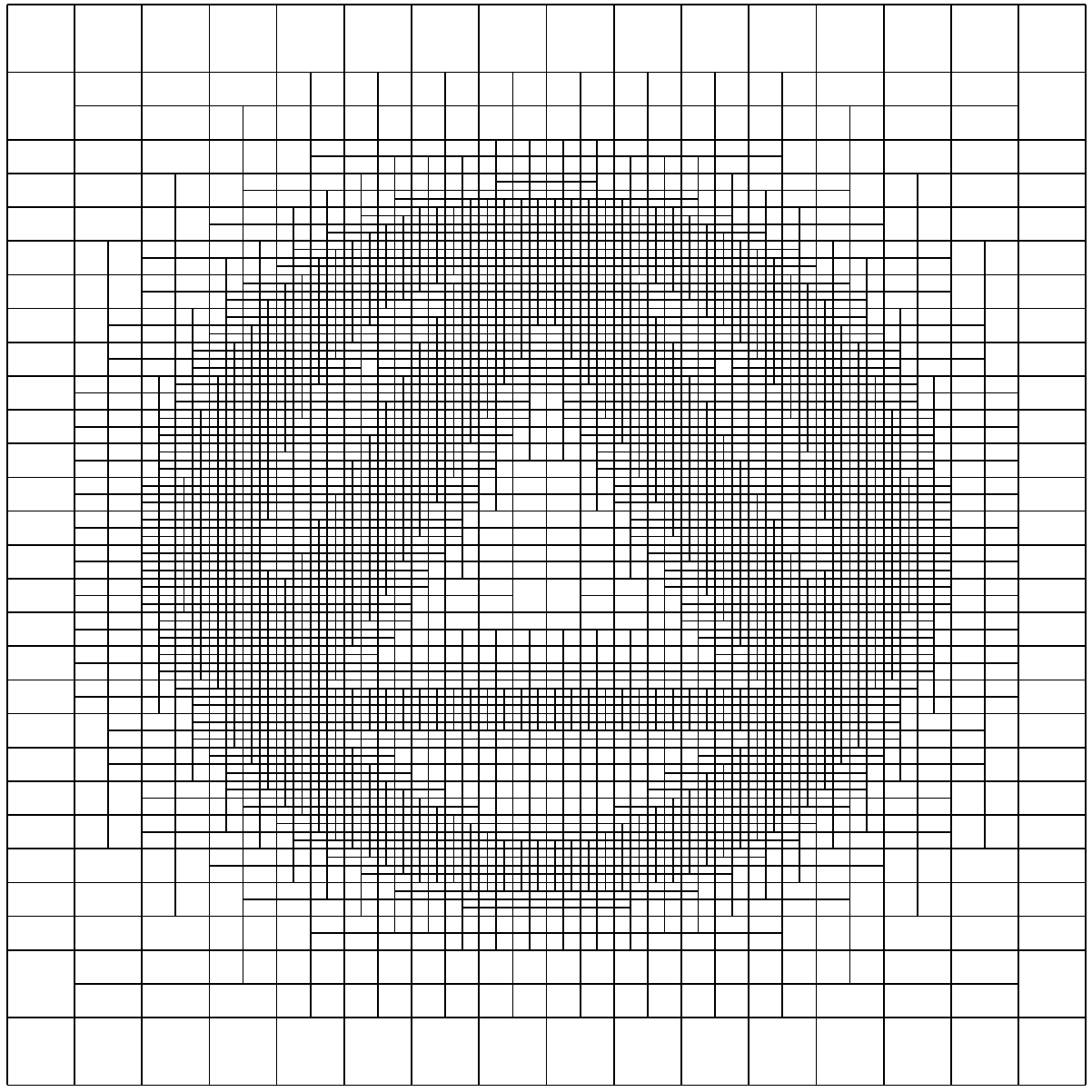}\pgfmatrixnextcell\includegraphics[width=.3\textwidth]{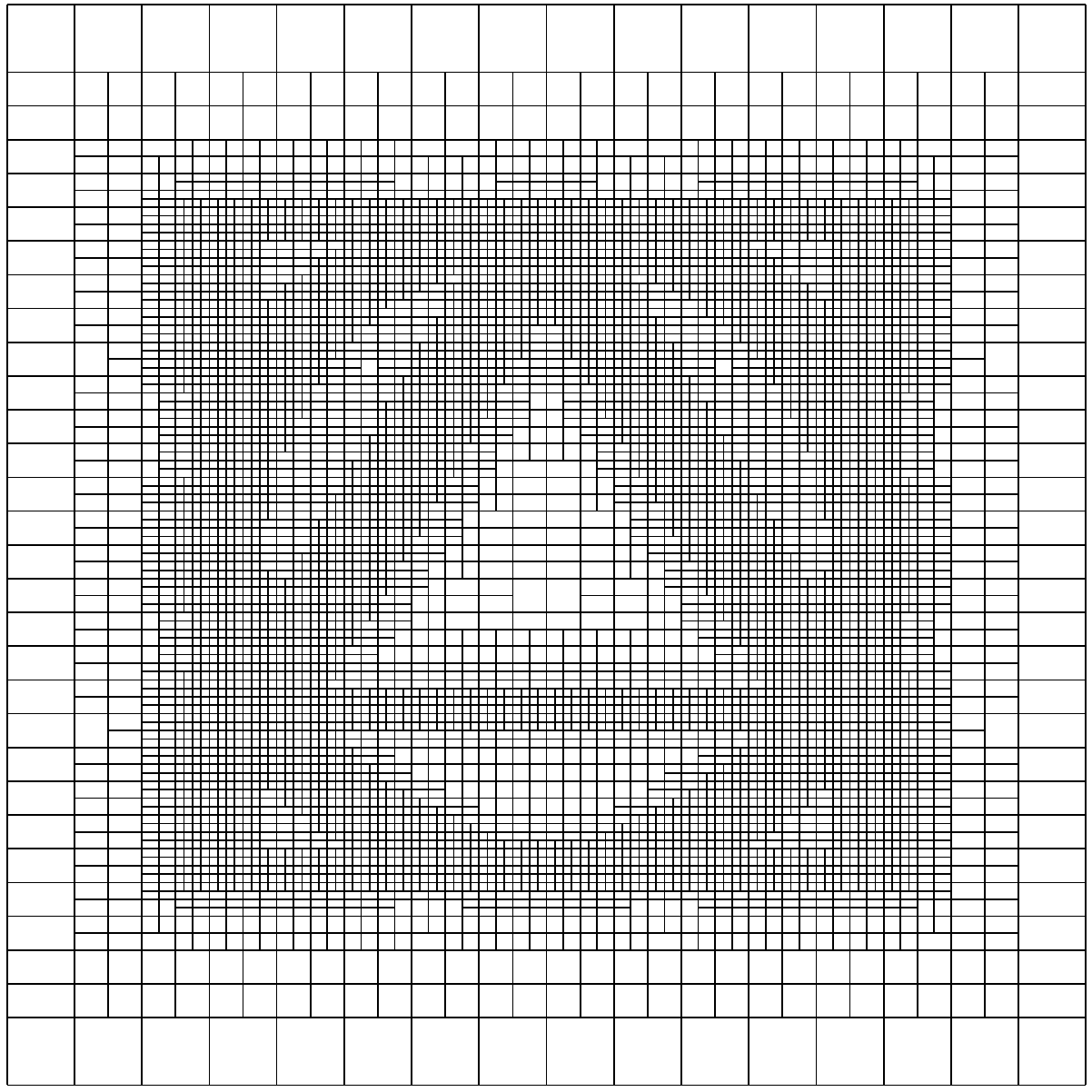}\\
};

\draw[\freccia] (m-1-1) -- (m-1-2);
\draw[\freccia] (m-1-2) -- (m-1-3);
\draw[\freccia] (m-1-3) -- (m-2-3);
\draw[\freccia] (m-2-3) -- (m-2-2);
\draw[\freccia] (m-2-2) -- (m-2-1);
\draw[\freccia] (m-2-1) -- (m-3-1);
\draw[\freccia] (m-3-1) -- (m-3-2);
\draw[dashed,\freccia] (m-3-2) -- (m-3-3);
\end{tikzpicture}
\caption{Example showing the adaptivity of the EG strategy. From a refinement localized along a triangle, we perform iterations on the the circumscribed circle and then on the square in which the circle is inscribed, switching the regions of refinement. The figure has to be read following the arrows which represent the iterations. For a matter of space, we do not show the intermediate steps when moving from the circle to the square to get the final mesh. The EG strategy guarantees local linear independence of the LR B-splines in all the iterations. The bidegree considered is $\pmb{p}=(2,2)$.}\label{fig:EGex2}
\end{figure}

\begin{oss}\label{remarkshadow}
We highlight the importance of refining only the closest to $\beta$ of the boxes in $\cS \beta$ of diameter larger than $sd$ in Algorithm \ref{alg:EGgrader} and updating the shadow. Halving only one box and updating the shadow avoids the presence of extra spurious lines in the mesh at the end of the process, as explained in Figure \ref{fig:oneattime}.
\end{oss}
\begin{figure}
\centering
\subfloat{
\begin{tikzpicture}
\node at (0,3.25) {\begin{tikzpicture}
\fill[rosso] (4,0) rectangle (4.5,1);
\fill[celeste] (0,0) rectangle (4,1);
\draw (0,0) rectangle (4.5,1);
\draw (2,0) -- (2,1);
\draw (4,0) -- (4,1);
\draw (4,.5) -- (4.5,.5);
\node at (2.25,-.35) {{\footnotesize (a)}};
\end{tikzpicture}};
\node at (0,1.5) {\begin{tikzpicture}
\fill[rosso] (4,0) rectangle (4.5,1);
\fill[celeste] (2,0) rectangle (4,1);
\draw (0,0) rectangle (4.5,1);
\draw (2,0) -- (2,1);
\draw (4,0) -- (4,1);
\draw (2,.5) -- (4.5,.5);
\draw (3,0) -- (3,1);
\node at (2.25,-.35) {{\footnotesize (b)}};
\end{tikzpicture}};
\node at (0,-.25) {\begin{tikzpicture}
\fill[rosso] (4,0) rectangle (4.5,1);
\fill[celeste] (2,0) rectangle (4,1);
\draw (0,0) rectangle (4.5,1);
\draw (2,0) -- (2,1);
\draw (4,0) -- (4,1);
\draw (2,.5) -- (4.5,.5);
\draw (3,0) -- (3,1);
\draw[red] (1,0) -- (1,1);
\node at (2.25,-.35) {{\footnotesize (c)}};
\end{tikzpicture}};
\end{tikzpicture}
}
\raisebox{.65cm}{
\captionsetup[subfloat]{labelformat=empty}
\subfloat[(d)]{
\includegraphics[width=.3\textwidth]{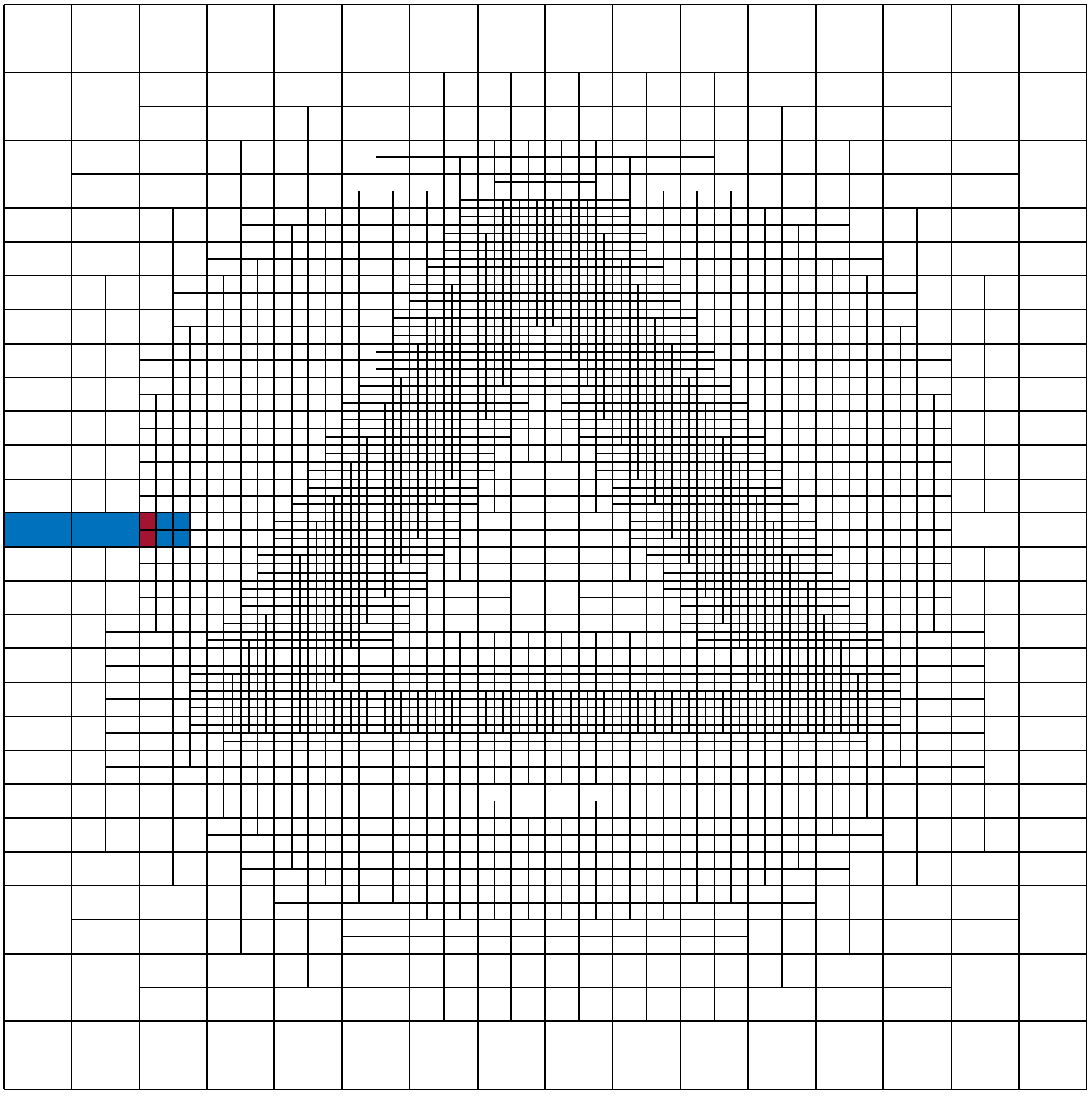}
}
\subfloat[(e)]{
\includegraphics[width=.3\textwidth]{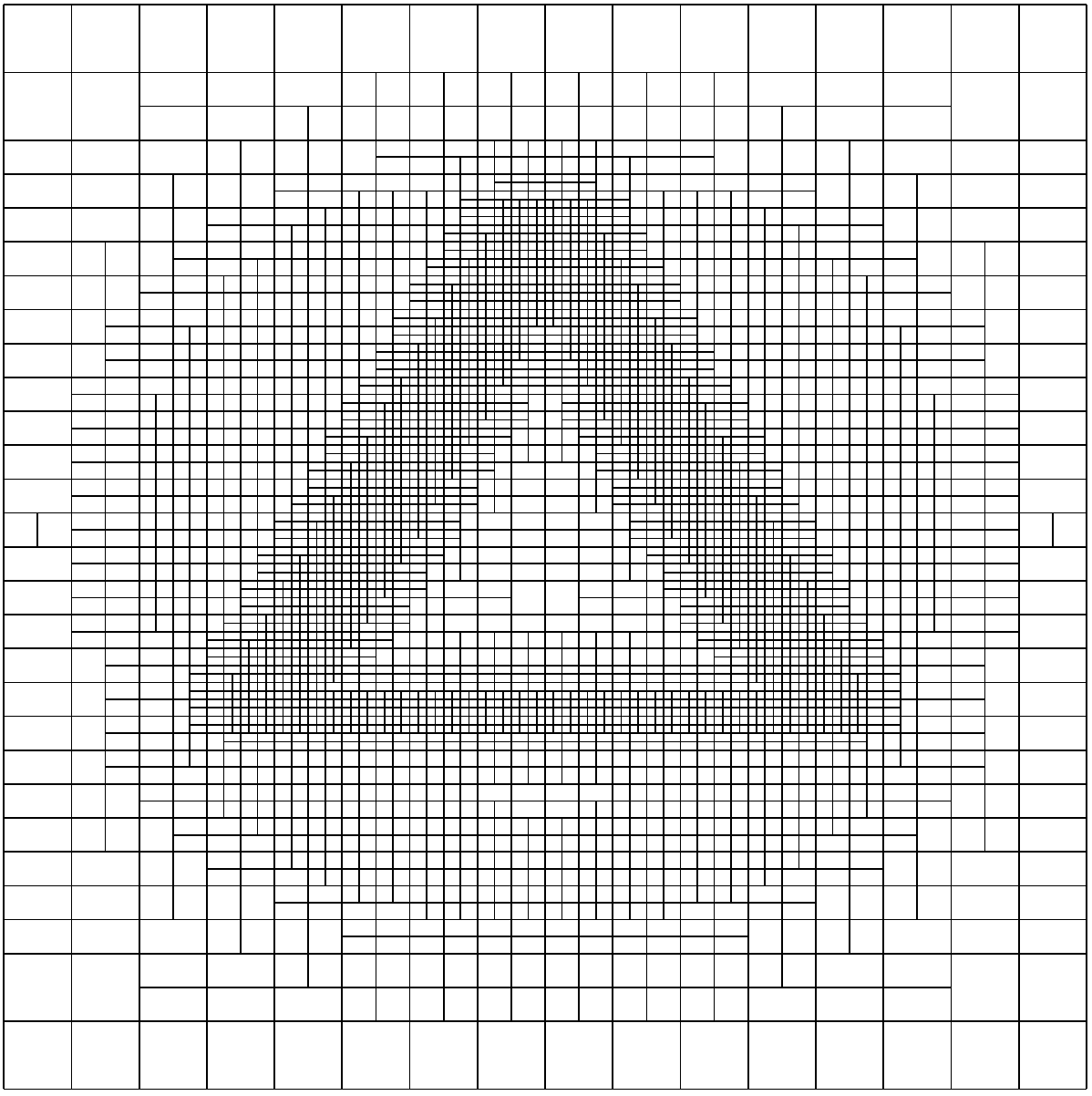}
}}
\caption{Supporting figure to Remark \ref{remarkshadow}. Let $p_1 =2$. In figure (a) we represent the horizontal generalized shadow of the two right-most boxes. The two boxes on the left in such shadow are too large for the EG strategy. If we proceed as described in Algorithm \ref{alg:EGgrader}, we halve only the closest to the boxes considered and we update the shadow before performing a further refinement. The result is represented in figure (b). If instead we halve both the large boxes in the shadow shown in figure (a) we have the mesh in figure (c) at the end of the process. As one can see, in figure (c) there is an extra vertical line. This vertical line may make the mesh be not an LR mesh anymore as it happens in figures (d)--(e). Figure (d) is an intermediate step while performing Algorithm \ref{alg:EGgrader} in the refining process shown in Figure \ref{fig:EGex2}. In particular, in this stage we are checking that the shadow highlighted is composed of boxes of the right size. The two boxes on the left are too large for the EG strategy. If we halve both of them before updating the shadow we obtain the mesh in figure (e) at the end of Algorithm \ref{alg:EGgrader}. The small vertical line has not traversed any LR B-spline on the mesh in figure (d). Hence, the mesh in figure (e) is not an LR mesh anymore. Instead, if each time we refine only the closest box, we obtain the LR mesh at the center of Figure \ref{fig:EGex2}, which is an LR mesh with the \NS~property.}\label{fig:oneattime}
\end{figure}

\subsection{Grading and spanning properties}
In this section we present the further properties of the EG strategy. We first analyze the grading of the mesh and then we identify the space spanned by the related set of LR B-splines. More specifically, we show
\begin{itemize}
\item bounds on the thinning of the boxes throughout the refinement,
\item bounds on the size ratio of adjacent boxes,
\item that the space spanned fills up the ambient space of the spline functions on the LR mesh.
\end{itemize} 
Assume that $\cN=(\cM,\pmb{p},\mu)$ is an LR mesh built using the EG strategy schematized in Algorithm \ref{alg:EG}. Then the aspect ratio of a box of $\cM$ is either $1:1$ or $2:1$ as rectangular boxes, of aspect ratio $2:1$, are obtained from square boxes and vice-versa throughout the making of the mesh. 
Furthermore, we note that, because of the constraints imposed in the \NS~property restoring step of the strategy, reported in Algorithm \ref{alg:EGgrader}, a box of size $c_1\times c_2$ in $\cM$ can be side by side only with boxes of same size or size double/half in one or both dimensions, i.e., boxes of sizes $c_1\times c_2, (2^{\pm 1}c_1)\times c_2, c_1\times (2^{\pm 1}c_2)$ and $2^{\pm 1}(c_1 \times c_2)$. More precisely, along the direction of the generalized shadow map, which is established by the shape of the box, there will only be boxes of the same size or with a scaling factor $2$ in one of the two dimension. In the direction orthogonal to the shadow, we may find boxes of same size or of  size double/half in both dimensions.
These bounds on the box sizes and neighboring boxes avoid the thinning throughout the refinement process and guarantee smoothly grading transitions between finer and coarser regions of the LR meshes produced by the EG strategy. In particular, given two adjacent boxes $\beta, \beta'$ of $\cM$, called $h_\beta$ the square root of the area of $\beta$, it holds
\begin{equation*}
\begin{array}{lrl}
\frac{\diam(\beta)}{h_\beta} = \left\{\begin{array}{ll}\sqrt{2} &\text{for square boxes,}\\\\\sqrt{\frac{5}{2}}&\text{for rectangular boxes,} \end{array}\right.&\Rightarrow\frac{\diam(\beta)}{h_\beta}\leq \sqrt{\frac{5}{2}},&\text{(A1)}\\\\
\frac{h_\beta}{h_{\beta'}} = \left\{\begin{array}{ll}
1 & \text{if $\beta,\beta'$ have same width in both directions,}\\\\
2 & \text{if $\beta$ has width double that of $\beta'$ in both directions,}\\\\
\sqrt{2} & \text{if $\beta$ has width double that of $\beta'$ in only one direction,}
\end{array}\right.&\Rightarrow  \frac{h_\beta}{h_{\beta'}}\leq 2.&\text{(A2)}
\end{array}
\end{equation*}
Inequalities (A1)--(A2) show that the box-partition associated to $\cM$ satisfies the \textbf{shape regularity} and \textbf{local quasi uniformity} conditions, which are two of the so-called \textbf{axioms of adaptivity}: a set of requirements which theoretically ensure optimal algebraic convergence rate in adaptive FEM and IgA, see \cite{axioms} and \cite[Sections 5--6]{b} for details. In particular, conditions (A1)--(A2) is what is demanded in terms of grading and overall appearance of the mesh used for the discretization. 
 
We now prove another important feature of the EG strategy: the space spanned is the entire \textbf{spline space}. The spline space on a given LR mesh $\cN=(\cM, \pmb{p},\mu)$, denoted by $\SSS(\cN)$, is defined as 
\begin{equation*}
\SSS(\cN) := \left\{\begin{split}
f:\RR^2 \to \RR \,:\,& \supp f \subseteq \Omega,\\
& f|_{\beta} \text{ is a polynomial of bidegree }\pmb{p}\text{ in any }\beta\text{ box of }\cM,\\
& f \in C^{p_{3-k}-\mu(\gamma)}\text{-continuous across any meshline }\gamma\text{ of }\cM\text{ along the }k\text{th direction}.
\end{split}\right\}.
\end{equation*}
In general, all the spaces spanned by generalizations of the B-splines addressing adaptivity, such as LR spline spaces, are just subspaces of the spline space on the underlying mesh. The next result ensures that when we are using LR meshes generated by the EG strategy, the span of the LR B-splines actually fills up the entire spline space.
\begin{thm}
Let $\cN = (\cM, \pmb{p}, \mu)$ be an LR-mesh provided by several iterations of EG strategy and let $\cL$ be the associated LR B-spline set. Then $\spn \cL = \SSS(\cN)$.
\end{thm}
\begin{proof}
If $\cN$ is a tensor mesh, the LR B-spline set coincides with the tensor B-spline set and the statement is true by the Curry-Schoenberg Theorem. If instead there are local insertions in $\cM$, we recall that during the refining steps of the EG strategy that yielded $\cM$, we have always inserted new lines  traversing the support of at least one LR B-spline. This means that each new line along the $k$th direction traversed at least $p_k+2$ orthogonal meshlines when has been inserted. In the \NS~property recovering steps we then have further prolonged some of such lines. By \cite[Theorem 12]{bressan2}, this ``length'' of the lines in terms of intersections guarantees that $\spn\cL = \SSS(\cN)$.
\end{proof}
This spanning property is achieved also by using the HLR strategy \cite{bressan2}.
\section{Conclusion}\label{conclusions}
We have presented a simple refinement strategy ensuring the local linear independence of the associated LR B-splines. Furthermore, the width of the regions refined at each iteration of the strategy guarantees that the span of the LR B-splines fills up the whole spline space on the LR mesh. 

We have called it Effective Grading (EG) strategy as the transition between coarser and finer regions is rather gradual and smooth in the LR meshes produced, with strict bounds on the aspect ratio of the boxes and on the sizes of the neighboring boxes. Such a grading ensures that the requirements on the mesh appearance listed in the axioms of adaptivity \cite{axioms,b} are verified. The latter are a set of sufficient conditions on mesh grading, refinement strategy, error estimates and approximant spaces in adaptive numerical methods to theoretically guarantee optimal algebraic convergence rate of the numerical solution to the real solution. The verification of the remaining axioms will be the topic of future research.

\section*{Acknowledgments}
This work was partially supported by the European Council under the Horizon 2020 Project Energy oriented Centre of Excellence for computing applications - EoCoE, Project ID 676629.
The author is member of Gruppo Nazionale per il Calcolo Scientifico, Istituto Nazionale di Alta Matematica.

\begin{appendices}
\section{Equivalence with the shadow map}
In this appendix we show that the generalized shadow map is equivalent to the definition of shadow map, given in \cite[Definition 10]{bressan2}, when the underlying mesh is a tensor mesh and the set considered is constituted of a collection of boxes. In order to recall the latter, we introduce the \textbf{separation distance}. Given a direction $k \in \{1,2\}$, let $\cN=(\cM,\pmb{p},\mu)$ be a tensor mesh and $\cM_{3-k}$ be the subcollection in $\cM$ of all the meshlines in the $(3-k)$th direction. Given two points $\pmb{p},\pmb{q} \in \Omega$ the separation distance of $\pmb{p}$ and $\pmb{q}$ along direction $k$ with respect to the tensor mesh $\cN$ is defined as 
$$
\spp_k^{\cN} (\pmb{p},\pmb{q}) := \left\{\begin{array}{ll}
\#\{\pmb{t} \in \Omega \,|\, \pmb{t} \in \gamma \cap \cM_{3-k}\} & \text{if }\pmb{p}, \pmb{q}\text{ are axis-aligned in the $k$th direction},\\
+\infty &\text{otherwise,}
\end{array}\right.
$$
with $\gamma$ the segment along direction $k$ between $\pmb{p}$ and $\pmb{q}$.  
Given a set $A\subseteq \Omega$ composed of boxes of $\cM$, we define 
$$
\spp_k^{\cN}(\pmb{p},A) = \inf_{\pmb{q} \in A} \spp_k^{\cN}(\pmb{p},\pmb{q}).
$$
The definition of shadow map along direction $k$ with respect to the tensor mesh $\cN$ given in \cite{bressan2} is then
$$
\fS A = \overline{\{\pmb{p} \in \Omega \,|\, \spp_k^{\cN}(\pmb{p},A) \leq p_k\}}.
$$
We use the gothic symbol $\fS$ to distinguish it from the generalized shadow map, defined in Section \ref{sec:EGpre}. We now show that the two are equivalent. Let $\beta$ be a box in $A$. Then, $\spp_k^{\cN}(\pmb{p},A) = 0$  for all $\pmb{p} \in \beta$ and $\beta \subseteq \fS(A)$. Let instead $\beta \subseteq \cS A\setminus A$. Then $\spp_k^{\cN}(\pmb{p},A) \leq p_k+1$ for all $\pmb{p} \in \beta$ and so $\beta \subseteq \fS A$. Note that the ``for all $\pmb{p} \in \beta$'' is true because $A$ is composed of boxes of $\cM$. Otherwise, there could be points $\pmb{p}$ in $\beta$ that are not part of the shadow $\fS A$, see, e.g., \cite[Figure 5]{bressan2}. We have proved that  $\cS A \subseteq \fS A$. We now show the opposite, that is, $\cS A \supseteq \fS A$.  Let $\pmb{p} \in \fS A$. Then $p_k +1 \geq \inf_{\pmb{q} \in A} \spp_k^{\cN}(\pmb{p},\pmb{q})$. If $\pmb{p} \notin A$ then such infimum is reached for some $\pmb{q} \in \partial A$ and so $\pmb{p} \in \overline{\pmb{q}_*^1\pmb{q}_*^2}$, with $\pmb{q}_*^1$ as defined in Equation \eqref{q*}. Therefore $\pmb{p}$ is contained in a box $\beta$ of $\cM$ intersecting $\overline{\pmb{q}_*^1\pmb{q}_*^2}$ and $\beta \subseteq \cS A$. If instead $\pmb{p} \in A$, there is nothing to prove as $A$ is in $\cS A$. 
\end{appendices}

\bibliography{biblio}

\providecommand{\bysame}{\leavevmode\hbox to3em{\hrulefill}\thinspace}
\providecommand{\MR}{\relax\ifhmode\unskip\space\fi MR }
\providecommand{\MRhref}[2]{%
  \href{http://www.ams.org/mathscinet-getitem?mr=#1}{#2}
}
\providecommand{\href}[2]{#2}
\begin{thebibliography}{10}

\bibitem{ast1}
Lourenco Beir\~{a}o~da Veiga, Annalisa Buffa, Giancarlo Sangalli, and Rafael
  V\'{a}zquez, \emph{Mathematical analysis of variational isogeometric
  methods}, Acta Numer. \textbf{23} (2014), 157--287. \MR{3202239}

\bibitem{bressan1}
Andrea Bressan, \emph{Some properties of {LR}-splines}, Comput. Aided Geom.
  Design \textbf{30} (2013), no.~8, 778--794. \MR{3146870}

\bibitem{bressan2}
Andrea Bressan and Bert J\"{u}ttler, \emph{A hierarchical construction of {LR}
  meshes in 2{D}}, Comput. Aided Geom. Design \textbf{37} (2015), 9--24.
  \MR{3370382}

\bibitem{b}
Annalisa Buffa, Gregor Gantner, Carlotta Giannelli, Dirk Praetorius, and Rafael
  V{\'a}zquez, \emph{Mathematical foundations of adaptive isogeometric
  analysis}, arXiv preprint arXiv:2107.02023 (2021).

\bibitem{axioms}
Carsten Carstensen, Michael Feischl, Marcus Page, and Dirk Praetorius,
  \emph{Axioms of adaptivity}, Comput. Math. Appl. \textbf{67} (2014), no.~6,
  1195--1253. \MR{3170325}

\bibitem{ast}
Lourenco Beir\~{a}o Da~Veiga, Annalisa Buffa, Giancarlo Sangalli, and Rafael
  V\'{a}zquez, \emph{Analysis-suitable {T}-splines of arbitrary degree:
  definition, linear independence and approximation properties}, Math. Models
  Methods Appl. Sci. \textbf{23} (2013), no.~11, 1979--2003. \MR{3084741}

\bibitem{deboor}
Carl de~Boor, \emph{A practical guide to splines}, Applied Mathematical
  Sciences, vol.~27, Springer-Verlag, New York-Berlin, 1978. \MR{507062}

\bibitem{ast2}
Jiansong Deng, Falai Chen, and Yuyu Feng, \emph{Dimensions of spline spaces
  over {$T$}-meshes}, J. Comput. Appl. Math. \textbf{194} (2006), no.~2,
  267--283. \MR{2239393}

\bibitem{pht}
Jiansong Deng, Falai Chen, Xin Li, Changqi Hu, Weihua Tong, ZZhouwang Yang, and
  Yuyu Feng, \emph{Polynomial splines over hierarchical {T}-meshes}, Graphical
  Models \textbf{70} (2008), 76--86.

\bibitem{tor}
Tor Dokken, Tom Lyche, and Kjell~Fredrik Pettersen, \emph{Polynomial splines
  over locally refined box-partitions}, Comput. Aided Geom. Design \textbf{30}
  (2013), no.~3, 331--356. \MR{3019748}

\bibitem{hb}
David~R. Forsey and Richard~H. Bartels, \emph{Hierarchical {B}-spline
  refinement}, ACM Siggraph Computer Graphics \textbf{22} (1988), 205--212.

\bibitem{thb2}
Carlotta Giannelli and Bert J\"{u}ttler, \emph{Bases and dimensions of
  bivariate hierarchical tensor-product splines}, J. Comput. Appl. Math.
  \textbf{239} (2013), 162--178. \MR{2991965}

\bibitem{thb}
Carlotta Giannelli, Bert J\"{u}ttler, and Hendrik Speleers,
  \emph{T{HB}-splines: the truncated basis for hierarchical splines}, Comput.
  Aided Geom. Design \textbf{29} (2012), no.~7, 485--498. \MR{2925951}

\bibitem{hendrik}
Clemens Hofreither, Ludwig Mitter, and Hendrik Speleers, \emph{Local multigrid
  solvers for adaptive isogeometric analysis in hierarchical spline spaces},
  IMA Journal of Numerical Analysis (2021).

\bibitem{iga}
Thomas J.~R. Hughes, John~A. Cottrell, and Yuri Bazilevs, \emph{Isogeometric
  analysis: {CAD}, finite elements, {NURBS}, exact geometry and mesh
  refinement}, Comput. Methods Appl. Mech. Engrg. \textbf{194} (2005),
  no.~39-41, 4135--4195. \MR{2152382}

\bibitem{johannessen}
Kjetil~Andr\'{e} Johannessen, Trond Kvamsdal, and Tor Dokken,
  \emph{Isogeometric analysis using {LR} {B}-splines}, Comput. Methods Appl.
  Mech. Engrg. \textbf{269} (2014), 471--514. \MR{3144651}

\bibitem{manni1}
Tom Lyche, Carla Manni, and Hendrik Speleers, \emph{Foundations of spline
  theory: {B}-splines, spline approximation, and hierarchical refinement},
  Splines and {PDE}s: from approximation theory to numerical linear algebra,
  Lecture Notes in Math., vol. 2219, Springer, Cham, 2018, pp.~1--76.
  \MR{3839186}

\bibitem{manni2}
Carla Manni and Hendrik Speleers, \emph{Standard and non-standard {CAGD} tools
  for isogeometric analysis: a tutorial}, Isogeometric analysis: a new paradigm
  in the numerical approximation of {PDE}s, Lecture Notes in Math., vol. 2161,
  Springer, [Cham], 2016, pp.~1--69. \MR{3586483}

\bibitem{thb1}
Dominik Mokri\v{s}, Bert J\"{u}ttler, and Carlotta Giannelli, \emph{On the
  completeness of hierarchical tensor-product {$B$}-splines}, J. Comput. Appl.
  Math. \textbf{271} (2014), 53--70. \MR{3209913}

\bibitem{nochetto}
Ricardo~H. Nochetto and Andreas Veeser, \emph{Primer of adaptive finite element
  methods}, Multiscale and adaptivity: modeling, numerics and applications,
  Lecture Notes in Math., vol. 2040, Springer, Heidelberg, 2012, pp.~125--225.
  \MR{3076038}

\bibitem{lindep}
Francesco Patrizi and Tor Dokken, \emph{Linear dependence of bivariate minimal
  support and locally refined {B}-splines over {LR}-meshes}, Comput. Aided
  Geom. Design \textbf{77} (2020), 101803, 22. \MR{4046412}

\bibitem{N2S2}
Francesco Patrizi, Carla Manni, Francesca Pelosi, and Hendrik Speleers,
  \emph{Adaptive refinement with locally linearly independent {LR} {B}-splines:
  theory and applications}, Comput. Methods Appl. Mech. Engrg. \textbf{369}
  (2020), 113230, 20. \MR{4118824}

\bibitem{schumaker}
Larry~L. Schumaker, \emph{Spline functions: basic theory}, third ed., Cambridge
  Mathematical Library, Cambridge University Press, Cambridge, 2007.
  \MR{2348176}

\end{thebibliography}
\end{document}